\documentclass{revtex4}
\usepackage{CJK}
\usepackage[utf8]{inputenc}
\usepackage{amssymb}
\usepackage{amsthm}
\usepackage{amsfonts}
\usepackage{amsmath}
\usepackage[normalem]{ulem}
\usepackage{tikz}
\usepackage{hyperref}
\usepackage[]{pstricks}
\usepackage{mathrsfs}
\usepackage[]{epsfig}
\usepackage{indentfirst}
\everymath{\displaystyle}
\newtheorem{remark}{Remark}

\newtheorem{lem}{Lemma}
\newtheorem{teo}{Theorem}
\newtheorem{prop}{Proposition}

\newtheorem{defi}{Definition}

\begin{document}

\title{The Cauchy Problem for nonlinear Quadratic Interactions of the Schrödinger type in one dimensional space}

\author{Isnaldo Isaac Barbosa\\ \footnotesize{Instituto de Matemática}\\
	\footnotesize{Universidade Federal de Alagoas}\\
	\footnotesize{Maceió, Alagoas}\\
	\footnotesize{isnaldo@pos.mat.ufal.br}}


\begin{abstract}
	 In this work I study the well-posedness of the Cauchy problem associated with the  coupled Schr\"odinger equations {with quadratic nonlinearities}, which appears modeling problems in nonlinear optics. I obtain the local well-posedness for data {in Sobolev spaces} with low regularity. {To obtain} the local theory, I prove new bilinear estimates for the coupling terms of the system in the continuous case. Concerning global results, in the continuous case, I establish the global well-posedness in $H^s(\mathbb{R})\times H^s(\mathbb{R})$, for some negatives indexes $s$.  The proof of the global result uses the \textbf{I}-method introduced by Colliander, Keel, Staffilani, Takaoka and Tao.  
\end{abstract}

\maketitle

\section{Introduction}

This work is dedicated to the study of the Cauchy problem for a system {that appears modeling some  problems in the context of nonlinear optics}. More precisely, we will study the following mathematical model:
\begin{equation}\label{1.a}
	\begin{cases}
		i\partial_t u(x,t)+p\partial^2_x u(x,t) -\theta u(x,t)+ \bar{u}(x,t)v(x,t)=0, & x\in \mathbb{R},\; t\ge 0,\\
		i\sigma\partial_t v(x,t)+q\partial^2_x v(x,t) -\alpha v(x,t)+\tfrac{a}{2}u^2(x,t)=0,&  \\
		u(x,0)=u_0(x),\quad  v(x,0)=v_0(x),
	\end{cases}
\end{equation}
where $u$ and $v$ are complex valued functions {and}  $\alpha$, $\theta$ and {$a: =1/\sigma$ are real numbers representing  physical parameters of the system, where $\sigma >0$ and $p, \ q\ =\pm1$. The model \eqref{1.a} is given by the nonlinear coupling of two dispersive equations of Schrödinger type through the quadratic terms
	\begin{equation}
		N_1(u,v)=\overline{u}\cdot v \ \mbox{ and } N_2(u,v)=\frac{1}{2}u\cdot v.
	\end{equation}
}

Physically, according to the article {\cite{menyuk-1994}}, 
the complex functions $u$ and $v$ represent amplitude packets of the first and second harmonic of an optical wave, respectively. The values of $p$ and $q$ may be $1$ or $-1$, depending on the signals provided between the scattering/diffraction ratios and the positive constant $\sigma$ measures the scaling/diffraction magnitude indices. In recent years, interest in nonlinear properties of optical materials has attracted attention of physicists and mathematicians. Many researches suggest that by exploring the nonlinear reaction of the matter, the bit-rate capacity of optical fibers can be considerably increased and in consequence an improvement in the speed and economy of data transmission and manipulation. Particularly in non-centrosimetric materials, those having no inversion symmetry at molecular level, the nonlinear effects of lower order give rise to second order susceptibility, which means that the nonlinear response to the electric field is quadratic; see, for instance, the articles \cite{karamzin-1974} and \cite{desalvo-1992}.

{Another application for the system \eqref{1.a} is related to the Raman amplification in a plasma}. The study of laser-plasma interactions is an active area of interest. The main goal is to simulate nuclear fusion in a laboratory. In order to simulate numerically these experiments, we need some accurate models. The kinetic ones are the most relevant but very difficult to deal with practical computations. The fluids ones like bifluid Euler–Maxwell system seem more convenient but still inoperative in practice because of the high frequency motion and the small wavelength involved in the problem. This is why we need some intermediate models that are reliable from a numerical viewpoint \cite{colin2009stability}.

In the mathematical context N. Hayashi, T. Ozawa, K. Tanaka in \cite{hayashi2013} obtained local well-posedness for the Cauchy problem (\ref{1.a}) on the spaces $L^2(\mathbb{R}^n)\times L^2(\mathbb{R}^n)$ for $n\leq 4$ and $H^1(\mathbb{R}^n)\times H^1(\mathbb{R}^n)$ for $n\leq 6$. In \cite{li2014recent} the time decay estimates of small solutions to the systems under the mass resonance condition in 2-dimensional space  was revised. The authors also showed the existence of wave operators and modified wave operators of the systems under some mass conditions in $n$-dimensional space, where $n\geq 2$, and  showed the existence of scattering operators and finite time blow-up of the solutions for the systems in higher dimensional spaces.

Regarding to qualitative properties of Cauchy problem solutions (\ref{1.a}), we know that in the case where $p = q = 1$ the system was studied by F. Linares and J. Angulo in \cite{pava-2007} for initial data $u_0, v_0$ in the same periodic Sobolev space $H^s(\mathbb{T})$. More precisely, they obtained local well-posedness results  in  $H^s(\mathbb{T})\times H^s(\mathbb{T})$ for all $s \geq 0$ and obtained global well-posedness in the space $L^2(\mathbb{T}) \times L^2(\mathbb{T})$  using the conservation of the mass by the flow of the system, {that is,} the following conservation law:
\begin{equation}
	E(u(t), v(t))=\int_{-\infty}^{+\infty}\bigl( |u|^2+2\sigma |v|^2\bigr)dx= E(u_0,v_0).
\end{equation}

\begin{remark}
	The authors also {observed} in Comment 2.3 of \cite{pava-2007} that results can be obtained for data with lower regularity when $\sigma$ is different from 1, including: well-posedness in {$H_{per}^s \times H_{per}^s$} for $s> -1/2$. Furthermore, in the same work, stability and instability results were established for certain classes of periodic pulses. Another work devoted to the study of the existence and stability of wave type pulses for this model is due to A. Yew (see \cite{yew-2000}).	
\end{remark}

The techniques used in \cite{pava-2007} to obtain the results of local well-posedness follow the ideas in \cite{kenig-1996}, developed by C. Kenig, G. Ponce and L. Vega, where the initial value problem for a Schrödinger equation with quadratic nonlinearities in both periodic and continuous domain is studied. {More precisely}, they considered the following initial value problem:
\begin{equation}\label{eq-nls-i}
	\begin{cases}
		iu_t+ \partial_x^2u= N_j(u, \bar{u}),& x\in \mathbb{R}  \mbox{ or } x\in \mathbb{T}, t\ge 0,\\
		u(x,0)=u_0(x),& 
	\end{cases}
\end{equation}
where $N_1(u,\bar{u})= u\bar{u}$, $N_2(u,\bar{u})=u^2$ and $N_3(u,\bar{u})=\bar{u}^2$. The authors considered initial data in the Sobolev space $H^s$. In the continuous case, they proved local well-posedness for $s > -1/4$ in the case $j = 1$ and for $s> -3/4$ in the cases {$j = 2,\ 3$}. In the periodic case, local well-posedness was obtained for $s \geq 0$ when $j = 1$ and for $s> -1/2$ when {$j = 2,\ 3$}. To prove the local theory, they {used the Fourier restriction method}, known in the literature as {$X^{s, b}$-spaces} and introduced by J. Bourgain in \cite{bourgain-1993}. In this functional space, {sharp} bilinear estimates were proved. {These estimates} combined with the Banach Fixed Point Theorem applied to the integral operator associated to (\ref{1.a}) allowed us to obtain the desired local solutions. The lack of a conservation law for (\ref{eq-nls-i}) does not allow to get global results in some space, as usual.

We note that the results {given} in \cite{kenig-1996} can be applied to the system (\ref{1.a}) in the case where $\sigma=1$. In this situation, it is not difficult to obtain the local well-posedness in $H^s\times H^s$ for $s > -1/4$. However, a natural question arises:

\begin{quote}
	{\it What would be the scenery of the local and global well-posedness of the system (\ref{1.a}) when $\sigma \neq 1$ and for initial data in Sobolev spaces, not necessarily with the same regularity?}
\end{quote}

In this work, we considerer the Cauchy problem (\ref{1.a})  with any $\sigma>0$ and initial data $(u_0, v_0)$ belonging to Sobolev spaces $H^{\kappa}(\mathbb{R}) \times H^s(\mathbb{R})$ to answer the previous question. As far as we know, the local well-posedness for the system (\ref{1.a}) in low regularity it is unknow.

We will follow the ideias developed by A. J. Corcho and C. Matheus in  \cite{corcho-2009}, where they 
{treated} the Schrödinger-Debye system, modelled by
\begin{equation}\label{SD-i}
	\begin{cases}
		iu_t+ \tfrac{1}{2}\partial_x^2u= uv,& x\in \mathbb{R},\; t\ge 0,\\
		\mu v_t + v =\pm|u|^2,              & \mu >0,\\
		u(x,0)=u_0(x),\quad  v(x,0)=v_0(x),
	\end{cases}
\end{equation}
which also has quadratic type nonlinearities and the authors developed a local and global theory in Sobolev spaces with different regularities. They used the method also based on obtaining {sharp} bilinear estimates for the coupling terms in suitable Bourgain spaces as well as the use of fixed point techniques. 

{Moreover, in the same work, global results were obtained} by using a technique known as \textbf{I}-method which was first implemented by J. Colliander, M. Keel, G. Staffilani, H. Takaoka and T. Tao in \cite{tao-2001}.

Before of enunciate the main results, we given the following definition.
\begin{defi} Given $\sigma >0$, we say that the Sobolev indice pair $(\kappa, s)$ verifies the hypotheses $H_{\sigma}$ if it satisfies one of the following conditions:
	\begin{enumerate}
		\item [a)] $|\kappa|-1/2\leq s<\min\{\kappa+1/2,\ 2\kappa+1/2\}$\; for\; $0<\sigma<2$;
		\item [b)] $\kappa =s \ge 0$\; for\; $\sigma =2$;
		\item [c)] $|\kappa|-1\leq s<\min\{\kappa+1, \ 2\kappa+1 \}$\; for\; $\sigma >2$.
	\end{enumerate}
	{We denote}  
	\begin{equation}\label{region-bcl}
		\mathcal{W}_{\sigma}:=\Bigl\{(\kappa, s) \in \mathbb{R}^2;\; (\kappa, s)\; \text{verify the hypothesis}\; H_{\sigma}\Bigr\}.
	\end{equation}
\end{defi}

Throughout the paper, we fix a cutoff function $\psi$ in $C^{\infty}_0$  such that $0 \leq \psi(t) \leq 1,$
\begin{equation}
	\psi(t)=
	\begin{cases}
		1, \ \ \mbox{ if } \ \ |t|\leq 1\\
		0,\ \ \mbox{ if }\ \ |t|\geq 2
	\end{cases}
\end{equation}
and $\psi_T(t)=\psi\left(\frac{t}{T}\right)$.

Our main  local well-posedness result   is the following  statement.

\begin{teo}
	For any $\sigma>0$ and $(u_0,v_0)\in H^{\kappa}\times H^{s}$ where the Sobolev index pair $(\kappa, s)$ verifying the hypothesis $H_{\sigma}$, there exist a positive time $T=T\left(\|u_0\|_{H^{\kappa}}, \|v_0\|_{H^s},\sigma\right)$ and an unique solution $(u(t), v(t))$ for the initial value problem (\ref{1.a}), satisfying
	\begin{equation}
		\psi_T(t)u\in X^{\kappa,\frac 12+} \ \ \mbox{ and } \ \ \psi_T(t)v\in X_{1/\sigma}^{\kappa,\frac 12+},
	\end{equation} 
	\begin{equation}
		u\in C\left([0,T]; H^{\kappa}(\mathbb{R})\right) \ \ \mbox{ and } \ \ v\in C\left([0,T]; H^{s}(\mathbb{R})\right).
	\end{equation}
	
	Moreover, the map $(u_0,v_0)\longmapsto (u(t), v(t))$ is locally Lipschitz from $H^{\kappa}(\mathbb{R})\times H^s(\mathbb{R})$ into $C\left([0,T]; H^{\kappa}(\mathbb{R})\times H^s(\mathbb{R})\right)$.
\end{teo}


\begin{center}\label{figure}
	\begin{figure}[h]
	\includegraphics[scale=0.3]{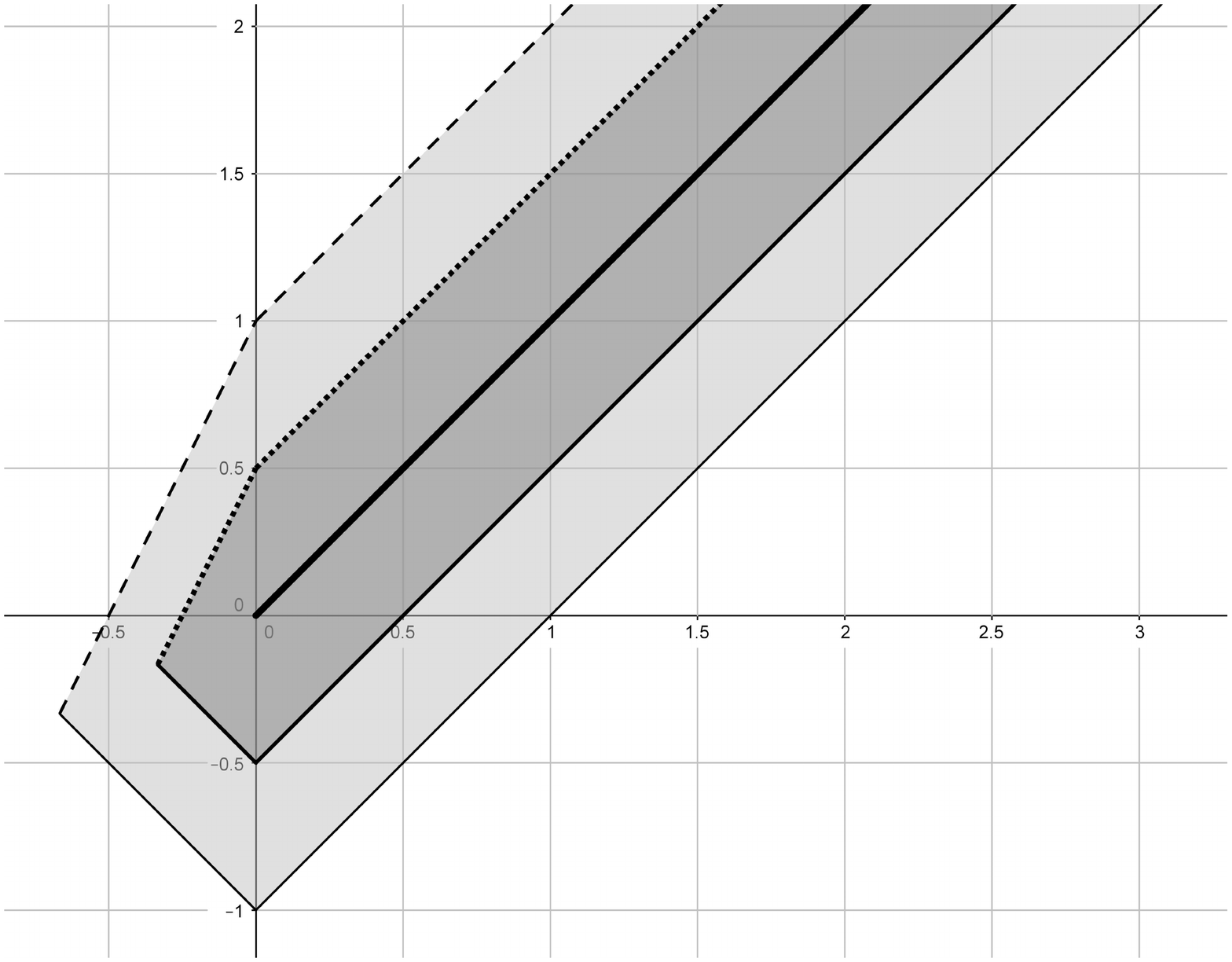}
	\caption{Region $\mathcal{W}_{\sigma}$}
	\label{R}
\end{figure}
\end{center}

Concerning global well-posedness we have the following result.

\begin{teo} \label{global}
	
	 In the following cases:
	\begin{itemize}
		\item $\sigma=2$ and $s=0$;
		\item $\sigma>2$ and $s\geq -1/2$;
		\item  $0<\sigma<2$ and $s\geq -1/4$.
	\end{itemize}
	  
	   The  Cauchy Problem associated to the system (\ref{1.a}) is globally well-posed, i.e., there exists a unique solution for any $T>0$ with initial condition  $(u_0,v_0)\in H^s(\mathbb{R})\times H^s(\mathbb{R})$.
	   
\end{teo}

Now we describe the structure of our work. {The Section 2 is devoted to  summarize some preliminary results}. In Section 3, we will develop a local theory in Bourgain spaces, following closely the techniques used in \cite{kenig-1996} and \cite{corcho-2009}, where for each positive $\sigma$ we obtain quite general results in Sobolev spaces with regularities {out of the diagonal case $\kappa=s$}. Specifically, we will prove {local well-posedness} for data $(u_0, v_0) \in H^{\kappa}\times H^s$ with indices $(\kappa, s)\in \mathcal{W}_{\sigma}$ (see the figure (\ref{figure})). 

Finally, in Section 4 we will use the {I-method to extend  globally} the local solutions obtained for data in $H^s\times H^s$ with for some negatives values of $s$. More precisely, we have regularity $-\frac{1}{4}\leq s\leq 0$ when $0<\sigma<2$ and $-\frac{1}{2}\leq s\leq 0$ when $\sigma>2$. At this point, it will be crucial the use of a refined Strichartz-type estimate in Bourgain's spaces for the Schrödinger equation. For details the reader can see \cite{colliander-2001g}.

\section{Preliminary results}

We consider the equation of the form
\begin{equation}\label{2.1}
	i\partial_t\omega-\phi\left(-i\partial_x\right)\omega=F(\omega),
\end{equation}
where $\phi$ is a measurable real-valued function and $F$ is a nonlinear function.

The Cauchy Problem for (\ref{2.1}) with initial data $\omega(0)=\omega_0$ is rewritten as the following integral equation
\begin{equation}\label{2.2}
	\omega(t)=W_{\phi}(t)\omega_0-i\displaystyle\int_0^t W_{\phi}(t-t')F(\omega(t'))dt',
\end{equation}
where $W_{\phi}(t)=e^{-it\phi(-i\partial_x)}$ is the group that solves the linear part of (\ref{2.1}).

Let $X^{s,b}(\phi)$ be the completion of $\mathcal{S}(\mathbb{R}^2)$ with respect to the norm
\begin{equation}\label{2.3}
	\begin{array}{ll}
		\left\| f \right\|_{X^{s,b}(\phi)}& :=\left\|W_{\phi}(-t)f \right\|_{H^b_t(\mathbb{R}, H^s_x)}\\
		& =\left\| \langle\xi\rangle^s\langle\tau\rangle^b \mathcal{F}\left(e^{it\phi(-i\partial_x)}f\right)(\tau,\xi) \right\|_{L^2_{\tau}L^2_{\xi}}\\
		&=\left\| \langle\xi\rangle^s\langle\tau+\phi(\xi)\rangle^b \widehat{f}(\tau,\xi) \right\|_{L^2_{\tau}L^2_{\xi}}.
	\end{array}
\end{equation}

The following lemma was proved while establishing the local well-posedness of the Zakharov system by Ginibre, Tsutsumi and Velo in \cite{ginibre-1997}.

\begin{lem}\label{l2.1}
	Let $-\frac{1}{2}<b'\leq 0\leq b\leq b'+1$, $\psi$ a cutoff function and $T\in [0,1]$. Then for $F\in X^{s, b'}(\phi)$ we have
	\begin{equation}\label{2.4}
		\left\|\psi_1(t)W_{\phi}(t)\omega_0 \right\|_{X^{s,b}(\phi)}\leq C \left\|\omega_0 \right\|_{H^s},
	\end{equation}
	\begin{equation}\label{2.5}
		\left\|\psi_T(t)\displaystyle\int_0^t W_{\phi}(t-t')F(\omega(t'))dt'\right\|_{X^{s,b}(\phi)}\leq C T^{1+b'-b} \left\|F \right\|_{X^{s,b'}(\phi)}.
	\end{equation}
\end{lem}

\begin{proof}
	See Lemma 2.1 in \cite{ginibre-1997}.
\end{proof}

In our case we shall use the space $X^{s,b}(\phi)$ for the phase functions $\phi_1(\xi)=\xi^2$ and $\phi_a(\xi)=a\xi^2.$ Indeed we can rewrite the system (\ref{1.a}) in the form
\begin{equation}\label{2.a}
	\begin{cases}
		i\partial_t u-\phi_1(-i\partial_x) u -\theta u+ \bar{u}v=0,&\\
		i\partial_t v-\phi_a(-i\partial_x) v -\alpha v+\tfrac{a}{2}u^2=0,&  \ a>0. \end{cases}
\end{equation} 
Then we have
\begin{equation*}
	X^{\kappa, b}(\phi_1)=X^{\kappa, b}, \ \ W_{\phi_1}=e^{it\partial^2_x}
\end{equation*}
and
\begin{equation*}
	X^{s, b}(\phi_a)=X^{s, b}_a, \ \ W_{\phi_a}=e^{iat\partial^2_x}.
\end{equation*}

We finish this section with the following elementary integral estimates which will be used to estimate the nonlinear terms in Section \ref{bilinear}.

\begin{lem}\label{l2.2}
	Let $p,q>0$, for $r=\mbox{min}\{p,q\}$ with $p+q>1+r$, there exists $C>0$ such that  
	\begin{equation}\label{conta1}
		\displaystyle\int_{\mathbb{R}}\dfrac{dx}{\langle x-\alpha\rangle^p \ \langle x-\beta \rangle^q}\leq \dfrac{C}{\langle \alpha-\beta \rangle^r}.
	\end{equation}
	Moreover, for $q>\frac{1}{2}$, 
	\begin{equation}\label{conta2}
		\displaystyle\int_{\mathbb{R}}\dfrac{dx}{\langle \alpha_0+\alpha_1 x+x^2\rangle^q}\leq C\ \mbox{ for all } \alpha_0,\alpha_1\in\mathbb{R}.
	\end{equation}
\end{lem}
\begin{proof}
	See Lemma 2.3 in \cite{kenig-1996}.
\end{proof}

\section{Bilinear estimates for the coupling terms}\label{bilinear}

 The main results in this section are the following propositions which present the bilinear estimates for different values of $\sigma>0$.  Each case lead us to different restrictions on the Sobolev indices $s$ and $\kappa$.

\subsection{Bilinear estimates for $\sigma>2$}

Next we prove a new bilinear estimates when $\sigma>2$ ($\sigma= 1/a$).
\begin{prop}\label{p1}
	{Let} $0<a<\frac{1}{2}$ (equivalently $\sigma >2$), $u\in X^{\kappa,b}$ and $v\in X_a^{s,b}$ with $1/2<b<3/4$, \   $1/4<d<1/2$ and $|\kappa|-s\leq 1$, then the bilinear estimate holds	
	\begin{equation}\label{eq1}
		\left\|\overline{u}\cdot v \right\|_{X^{\kappa,-d}}\leq C \left\|u \right\|_{X^{\kappa,b}}\cdot \left\|v \right\|_{X^{s,b}}. 
	\end{equation}
	
\end{prop}


The second result is the following
\begin{prop}\label{p2}
	{Let} $0<a<\frac{1}{2}$ (equivalently $\sigma >2$)  and  $u$, $ \tilde{u}\in X^{\kappa,b}$  with $1/2<b<3/4$,\   $1/4<d<1/2$ and  $s< \kappa +1$ if $\kappa \geq 0$ and $s<2\kappa +1$ if $\kappa <0$ then, the following estimate holds
	
	\begin{equation}\label{eq2}
		\left\|u\cdot \tilde{u} \right\|_{X^{s,-d}_a}\leq C \left\|u \right\|_{X^{\kappa,b}}\cdot \left\|\tilde{u} \right\|_{X^{\kappa,b}}.
	\end{equation}
	
\end{prop}

\medskip
\begin{proof}[Proof of the Proposition \ref{p1}]
	
	We define $$f(\xi,\tau)=\langle \tau-\xi^2\rangle^b\langle \xi\rangle^{\kappa}\widehat{\overline{u}}(\xi,\tau) \ \mbox{ and }\ g(\xi,\tau)=\langle \tau+a\xi^2\rangle^b\langle \xi\rangle^s\widehat{v}(\xi,\tau).$$
	
	Therefore, $\|f\|_{L^2_{\xi,\tau}}=\|u\|_{X^{{\kappa},b}}$ \ and\  $\|g\|_{L^2_{\xi,\tau}}=\|v\|_{X^{s,b}_a}$.\\
	It follows that,
	\begin{eqnarray*}
		\left\| \overline{u}\cdot v\right\|_{X^{{\kappa},-d}}& & = \left\| \langle \tau+\xi^2 \rangle^{-d} \langle \xi \rangle^{\kappa} \ \widehat{\overline{u}\cdot v}(\xi,\tau)\right\|_{L^2_{\xi,\tau}} \\
		&  &\hspace{-1cm}= \displaystyle\sup_{\|\varphi\|_{L^2_{\xi,\tau}\leq 1}} \left| \displaystyle\int_{\mathbb{R}^4} \dfrac{\langle \xi \rangle^{\kappa}\langle \xi_1 \rangle^{-{\kappa}}\langle \xi_2 \rangle^{-s}}{\langle\tau+\xi^2\rangle^{d}\langle\tau_1-\xi_1^2\rangle^b\langle\tau_2+a\xi_2\rangle^b}f(\xi_1,\tau_1)g(\xi_2,\tau_2)  \varphi(\xi,\tau)d\xi_2 d\tau_2 d\xi d\tau \right|.\\
	\end{eqnarray*}
	{We use the following notation:}
	\begin{equation}\label{3.2}
	\left\lbrace 
	\begin{array}{lll}
	\tau=\tau_1+\tau_2 & \xi=\xi_1+\xi_2&     \\
	\omega=\tau+\xi^2, & \omega_1=\tau_1-\xi_1^2,& \omega_2=\tau_2+a\xi_2^2
	\end{array}
	\right.
	\end{equation}
	{and we define}
	$$W(f,g,\varphi)= \displaystyle\int_{\mathbb{R}^4} \dfrac{\langle \xi \rangle^{\kappa}\langle \xi_1 \rangle^{-{\kappa}}\langle \xi_2 \rangle^{-s}}{\langle\omega\rangle^{d}\langle\omega_1\rangle^b\langle\omega_2\rangle^b}f(\xi_1,\tau_1)g(\xi_2,\tau_2)  \varphi(\xi,\tau)d\xi_2 d\tau_2 d\xi d\tau. $$
	{Now} it is suffices to prove that $$\left| W(f,g,\varphi)\right| \leq c \ \|f\|_{L^2}\cdot\|g\|_{L^2}\cdot\|\varphi\|_{L^2}.$$
	Consider $\mathbb{R}^4\subset \mathcal{R}_1\cup \mathcal{R}_2 \cup \mathcal{R}_3$, where $\mathcal{R}_j \subset \mathbb{R}^4$ for $j \in \{1,2,3\}.$ 
	We write 
	$$W_j=(f,g,\varphi)= \displaystyle\int_{\mathcal{R}_j} \dfrac{\langle \xi \rangle^{\kappa}\langle \xi_1 \rangle^{-{\kappa}}\langle \xi_2 \rangle^{-s}}{\langle\omega\rangle^{d}\langle\omega_1\rangle^b\langle\omega_2\rangle^b}f(\xi_1,\tau_1)g(\xi_2,\tau_2)  \varphi(\xi,\tau)d\xi_2 d\tau_2 d\xi d\tau$$
	and observe that $|W|\leq |W_1| +|W_2|+ |W_3|$.

	We estimate each case separately. {Using the Cauchy-Schwarz and H\"older inequalities and Fubini’s Theorem  we obtain}
	
	\begin{eqnarray*}
		\left| W_1\right|^2 & & =\left| \displaystyle\int_{\mathcal{R}_1} \frac{\langle \xi \rangle^{\kappa}\langle \xi_1 \rangle^{-\kappa}\langle \xi_2 \rangle^{-s}}{\langle\omega\rangle^{d}\langle\omega_1\rangle^b\langle\omega_2\rangle^b}f(\xi_1,\tau_1)g(\xi_2,\tau_2)  \varphi(\xi,\tau)d\xi_2 d\tau_2 d\xi d\tau\right|^2\\
		&& \hspace{-1.8cm} \leq \left\| f \right\|_{L^2}^2\left\| g \right\|_{L^2}^2\left\| \varphi \right\|_{L^2}^2 \left\| \frac{\langle \xi\rangle^{2{\kappa}}}{\langle\omega\rangle^{2d}}  \left(\displaystyle\int_{\mathbb{R}^2}\frac{\langle \xi_1\rangle^{-2{\kappa}} \langle \xi_2\rangle^{-2s}\chi_{\mathcal{R}_1}d\xi_2 d\tau_2}{\langle\omega_1\rangle^{2b}\langle\omega_2\rangle^{2b}}\right)\right\|_{L^{\infty}_{\xi,\tau}}.
	\end{eqnarray*}
	Similarly, we have
	
	\begin{eqnarray*}
		\left| W_2\right|^2 \leq \left\| f \right\|_{L^2}^2\left\|  g \right\|_{L^2}^2\left\| \varphi \right\|_{L^2}^2 \left\| \frac{\langle \xi_2\rangle^{2s}}{\langle\omega_2\rangle^{2b}}  \left(\displaystyle\int_{\mathbb{R}^2}\frac{\langle \xi_1\rangle^{-2{\kappa}} \langle \xi\rangle^{2{\kappa}}}{\langle\omega_1\rangle^{2b}\langle\omega\rangle^{2d}}\chi_{\mathcal{R}_2}d\xi d\tau\right)\right\|_{L^{\infty}_{\xi_2,\tau_2}}
	\end{eqnarray*}
	and 
	\begin{eqnarray*}
		\left| W_3\right|^2 \leq \left\| f \right\|_{L^2}^2\left\| g \right\|_{L^2}^2\left\| \varphi \right\|_{L^2}^2 \left\| \frac{\langle \xi_1\rangle^{-2{\kappa}}}{\langle\omega_1\rangle^{2b}}  \left(\displaystyle\int_{\mathbb{R}^2}\frac{\langle \xi\rangle^{2{\kappa}} \langle \xi_2\rangle^{-2s}}{\langle\omega\rangle^{2d}\langle\omega_2\rangle^{2b}}\chi_{\mathcal{R}_3}d\xi_2 d\tau_2\right)\right\|_{L^{\infty}_{\xi_1,\tau_1}}.
	\end{eqnarray*}
	{Using  Lemma \ref{conta2}  and the fact $\langle \xi\rangle^{2{\kappa}}\langle\xi_1 \rangle^{-2{\kappa}}\leq \langle \xi_2\rangle^{2|{\kappa}|}$, 
		we get the following inequalities:}
	\begin{equation*}
	\dfrac{\langle \xi\rangle^{2{\kappa}}}{\langle\omega\rangle^{2d}}  \displaystyle\int_{\mathbb{R}^2}\dfrac{\langle \xi_1\rangle^{-2{\kappa}} \langle \xi_2\rangle^{-2s}\chi_{\mathcal{R}_1}}{\langle\omega_1\rangle^{2b}\langle\tau_2+a \xi_2^2\rangle^{2b}}d\xi_2 d\tau_2\leq \underbrace{\dfrac{1}{\langle\omega\rangle^{2d}}\displaystyle\int_{\mathbb{R}}\dfrac{\langle\xi_2\rangle^{-2s+2|{\kappa}|}\chi_{\mathcal{R}_1}}{\langle\tau-(a-1)\xi_2^2-2\xi\xi_2+\xi^2\rangle^{2b}}d\xi_2}_{J_1},
	\end{equation*}
	
	\begin{equation*}
	\dfrac{\langle \xi_2\rangle^{2s}}{\langle\omega_2\rangle^{2b}}  \displaystyle\int_{\mathbb{R}^2}\dfrac{\langle \xi_1\rangle^{-2{\kappa}} \langle \xi\rangle^{2{\kappa}}\chi_{\mathcal{R}_2}}{\langle\omega_1\rangle^{2b}\langle\omega\rangle^{2d}}d\xi d\tau\leq \underbrace{\dfrac{1}{\langle\tau_2+a\xi_2^2\rangle^{2b}}\displaystyle\int_{\mathbb{R}}\dfrac{\langle\xi_2\rangle^{-2s+2|{\kappa}|}\chi_{\mathcal{R}_2}}{\langle\tau_2+2\xi^2+\xi_2^2-2\xi\xi_2\rangle^{2d}}d\xi}_{J_2}, 
	\end{equation*}
	
	\begin{equation*}
	\dfrac{\langle \xi_1\rangle^{-2{\kappa}}}{\langle\omega_1\rangle^{2b}}  \displaystyle\int_{\mathbb{R}^2}\dfrac{\langle \xi\rangle^{2\kappa} \langle \xi_2\rangle^{-2s}\chi_{\mathcal{R}_3}}{\langle\omega\rangle^{2d}\langle\omega_2\rangle^{2b}}d\xi_2 d\tau_2\leq\underbrace{ \dfrac{1}{\langle\tau_1-\xi_1^2\rangle^{2b}}\displaystyle\int_{\mathbb{R}}\dfrac{\langle\xi_2\rangle^{-2s+2|{\kappa}|}\chi_{\mathcal{R}_3}}{\langle\tau_1-a\xi_2^2+\xi^2\rangle^{2d}}d\xi_2}_{J_3}. 
	\end{equation*}

	It is enough to show that the functionals $J_1$, $J_2$ and $J_3$, defined below, are bounded.
	\begin{enumerate}
		\item[] \begin{equation}\label{j1} J_1= \dfrac{1}{\langle\tau+\xi^2\rangle^{2d}}\displaystyle\int_{\mathbb{R}}\dfrac{\langle\xi_2\rangle^{-2s+2|{\kappa}|}\chi_{\mathcal{R}_1}}{\langle\tau-(a-1)\xi_2^2-2\xi\xi_2+\xi^2\rangle^{2b}}d\xi_2;	\end{equation}
		\item[] \begin{equation}\label{j2} J_2= \dfrac{1}{\langle\tau_2+a\xi_2^2\rangle^{2b}}\displaystyle\int_{\mathbb{R}}\dfrac{\langle\xi_2\rangle^{-2s+2|{\kappa}|}\chi_{\mathcal{R}_2}}{\langle\tau_2+2\xi^2+\xi_2^2-2\xi\xi_2\rangle^{2d}}d\xi;	\end{equation}
		\item[] 	\begin{equation}\label{j3} J_3= \dfrac{1}{\langle\tau_1-\xi_1^2\rangle^{2b}}\displaystyle\int_{\mathbb{R}}\dfrac{\langle\xi_2\rangle^{-2s+2|{\kappa}|}\chi_{\mathcal{R}_3}}{\langle\tau_1-a\xi_2^2+\xi^2\rangle^{2d}}d\xi_2.	\end{equation}
	\end{enumerate}  

	
	In order to do so,  
	{we start by discussing} the dispersion of relations.
	Note that 
	\begin{eqnarray*}
		|\omega-\omega_1-\omega_2|&=& |\xi^2+\xi_1^2-a\xi_2^2|\\
		&\geq & |1-a|(\xi^2+\xi_1^2)-2a|\xi\xi_1|, \ \ \mbox{ suppose } 0<a<\frac{1}{2}\\
		&\geq & (1-a)(\xi^2+\xi_1^2)-a(\xi^2+\xi_1^2)=(1-2a)(\xi^2+\xi_1^2).
	\end{eqnarray*}
	It follows that, $$3\max \{|\omega|,|\omega_1|, |\omega_2|\}\geq (1-2a)\max\{\xi^2,\xi_1^2\}\geq \frac{1-2a}{4} \xi_2^2.$$
	Suppose that $|\xi_2|\geq 1$, {then} we have
	$$\dfrac{1}{\max \{|\omega|,|\omega_1|, |\omega_2|\}}\leq  \dfrac{c}{|\xi_2|^2}.$$
	Now, we define $\mathcal{R}_j$. 
	\begin{equation}
		\mathcal{R}_1= \bigg\lbrace |\xi_2|\geq 1, |\omega|=\max \{|\omega|,|\omega_1|, |\omega_2|\}\bigg\rbrace\cup\bigg\{|\xi_2| \leq 1\bigg\}\subset \mathbb{R}^4_{\xi,\tau,\xi_2,\tau_2};
	\end{equation}
	
	\begin{equation}
		\mathcal{R}_2= \bigg\{ |\xi_2|\geq 1, |\omega_1|=\max \{|\omega|,|\omega_1|, |\omega_2|\}\bigg\}\subset \mathbb{R}^4_{\xi,\tau,\xi_2,\tau_2};
	\end{equation}
	
	\begin{equation}
		\mathcal{R}_3= \bigg\{ |\xi_2|\geq 1, |\tau_2+a\xi^2|=\max \{|\omega|,|\omega_1|, |\omega_2|\}\bigg\}\subset \mathbb{R}^4_{\xi,\tau,\xi_2,\tau_2}.
	\end{equation}
	Let us prove that $J_1$ is bounded. Indeed, if $|\xi_2|\leq 1$ then $J_1$ is equivalent to $$\dfrac{1}{\langle\omega\rangle^{2d}}\displaystyle\int_{|\xi_2|\leq 1}\dfrac{1}{\langle\tau-(a-1)\xi_2^2-2\xi\xi_2+\xi^2\rangle^{2b}}d\xi_2\leq c.$$ If $|\xi_2|\geq 1$ then $J_1$ is bounded by  $$\displaystyle\int_{|\xi_2|\geq 1}\dfrac{\langle\xi_2\rangle^{-2s+2|{\kappa}|+4d}\chi_{\mathcal{R}_1}}{\langle\tau-(a-1)\xi_2^2-2\xi\xi_2+\xi^2\rangle^{2b}}d\xi_2.$$ Note that $J_1$ is bounded when $|{\kappa}|-s\leq 2d<1$ because $b>1/2$.
	
	To prove that $J_2$ is bounded, it is suffices to note that the integral below is higher than $J_2$ and that converges since $|{\kappa}|-s\leq 2b$ and that $2d>1/2 $, that is, $b<3/4$.
	$$ \displaystyle\int_{\mathbb{R}}\dfrac{\langle\xi_2\rangle^{-2s+2|{\kappa}|-4b}\chi_{\mathcal{R}_2}}{\langle\tau_2+2\xi^2+\xi_2^2-2\xi\xi_2\rangle^{2d}}d\xi.$$
	
	Analogously, a similar way, we can prove that $J_3$ is bounded, by using that $|{\kappa}|-s\leq 2b$ and $b<3/4$.

\end{proof}

Now we prove that the second non-linear term of the system is bounded.

\begin{proof}[Proof of the Proposition \ref{p2}]

	Analogously to the previous proposition, the estimate (\ref{eq2}) is equivalent
	to prove that the functionals $J_4$, $J_5$ and $J_6$, defined below, are bounded
		\begin{enumerate}
		\item[] \begin{equation}\label{j4} J_4= \dfrac{1}{\langle\lambda\rangle^{2d}}\displaystyle\int_{\mathbb{R}}\dfrac{\langle \xi\rangle^{2s}\langle \xi_1\rangle^{-2{\kappa}} \langle \xi_2\rangle^{-2{\kappa}}\chi_{\mathcal{S}_1}}{\langle\tau+\xi_2^2-2\xi\xi_2+\xi^2\rangle^{2b}}d\xi_2;\end{equation}
		\item[]\begin{equation}\label{j5} J_5= \dfrac{1}{\langle\lambda_2\rangle^{2b}}\displaystyle\int_{\mathbb{R}}\dfrac{\langle \xi\rangle^{2s}\langle \xi_1\rangle^{-2{\kappa}} \langle \xi_2\rangle^{-2{\kappa}}\chi_{\mathcal{S}_2}}{\langle\tau_2+(a-1)\xi^2-\xi_2^2+2\xi\xi_2\rangle^{2d}}d\xi;\end{equation}
		\item[] \begin{equation}\label{j6} J_6= \dfrac{1}{\langle\lambda_1\rangle^{2b}}\displaystyle\int_{\mathbb{R}}\dfrac{\langle \xi\rangle^{2s}\langle \xi_1\rangle^{-2{\kappa}} \langle \xi_2\rangle^{-2{\kappa}}\chi_{\mathcal{S}_3}}{\langle\tau_1+a\xi^2+\xi_2^2\rangle^{2d}}d\xi_2,\end{equation}
	\end{enumerate}
 where  $\mathcal{S}_1\cup\mathcal{S}_2\cup\mathcal{S}_3=\mathbb{R}^4$ with $\mathcal{S}_j$ measurable.
	
	Note that 
	\begin{eqnarray*}
		|\lambda-\lambda_1-\lambda_2|&=& |a\xi^2-\xi_1^2-\xi_2^2|\\
		&\geq & |1-a|(\xi_1^2+\xi_2^2)-2a|\xi_1\xi_2|, \ \ \mbox{ suppose } 0<a<\frac{1}{2}\\
		&\geq & (1-a)(\xi_1^2+\xi_2^2)-a(\xi_1^2+\xi_2^2)=(1-2a)(\xi_1^2+\xi_2^2),
	\end{eqnarray*}
	indeed $\xi=\xi_1+\xi_2$, such that $|\xi|\leq |\xi_1|+|\xi_2|\leq 2\max \{\xi_1,\xi_2\}.$ 
	Hence, $$3\max \{|\lambda|,|\lambda_1|, |\lambda_2|\}\geq (1-2a)\max\{\xi_1^2,\xi_2^2\}\geq \frac{1-2a}{4} \xi^2.$$
	Therefore, supposing that $|\xi|\geq 1$, we have
	$$\dfrac{1}{\max \{|\lambda|,|\lambda_1|, |\lambda_2|\}}\leq  \dfrac{c}{|\xi|^2}.$$
	Now, we define the regions $\mathcal{S}_i$. 
	\begin{equation}
		\mathcal{S}_1= \bigg\{ |\xi|\geq 1, |\lambda|=\max \{|\lambda|,|\lambda_1|, |\lambda_2|\}\bigg\}\cup\bigg\{|\xi| \leq 1\bigg\}\subset \mathbb{R}^4_{\xi,\tau,\xi_2,\tau_2},
	\end{equation}
	\begin{equation}
		\mathcal{S}_2= \bigg\{ |\xi|\geq 1, |\lambda_1|=\max \{|\lambda|,|\lambda_1|, |\lambda_2|\}\bigg\}\subset \mathbb{R}^4_{\xi,\tau,\xi_2,\tau_2},
	\end{equation}
	\begin{equation}
		\mathcal{S}_3= \bigg\{ |\xi|\geq 1, |\tau_2+\xi^2|=\max \{|\lambda|,|\lambda_1|, |\lambda_2|\}\bigg\}\subset \mathbb{R}^4_{\xi,\tau,\xi_2,\tau_2}.
	\end{equation}
	For $\kappa\geq 0$, we have 
	$\langle \xi_1 \rangle^{-2\kappa} \langle \xi_2 \rangle^{-2\kappa}\leq \langle \xi \rangle^{-2\kappa}$ and in this case 
	\begin{equation}
		J_4 \ \leq \displaystyle\int_{\mathbb{R}}\dfrac{\langle \xi\rangle^{2s-2{\kappa}+4d}\chi_{\mathcal{S}_1}}{\langle\tau+\xi_2^2-2\xi\xi_2+\xi^2\rangle^{2b}}d\xi_2.
	\end{equation}  
	Therefore, $J_4$ is bounded since $s-{\kappa}+2d<0$ for $s-{\kappa}<2d$.\\
	Note that $J_5$ and $J_6$ satisfies,
	\begin{equation}
		J_5 \ \leq \displaystyle\int_{\mathbb{R}}\dfrac{\langle \xi\rangle^{2s-2{\kappa}-4b}\chi_{\mathcal{S}_2}}{\langle\tau_2+(a-1)\xi^2-\xi_2^2+2\xi\xi_2\rangle^{2d}}d\xi
	\end{equation}  
	and
	\begin{equation}
		J_6 \ \leq \displaystyle\int_{\mathbb{R}}\dfrac{\langle \xi\rangle^{2s-2{\kappa}-4b}\chi_{\mathcal{S}_3}}{\langle\tau_1+a\xi^2+\xi_2^2\rangle^{2d}}d\xi_2,
	\end{equation}
	and that they are bounded since $s-{\kappa}<2b$ and $2d> \frac{1}{2}$, that is, $b<\frac{3}{4}.$

	When $\kappa<0$ we analyse the following subcases:
	\begin{enumerate}
		\item Considering $|\xi_1|\leq \frac{2}{3}|\xi_2|$: we have $\langle\xi_1\rangle^{-2\kappa}\langle\xi_2\rangle^{-2\kappa}\leq \langle \xi_2\rangle^{-4\kappa}$. Moreover, $|\xi_2|\leq |\xi_1|+|\xi|\leq \frac{2|\xi_2|}{3}+|\xi|$, hence $|\xi_2|\leq 3|\xi|$. Therefore, $$\langle \xi\rangle^{2s}\langle \xi_1\rangle^{-2\kappa} \langle \xi_2\rangle^{-2\kappa}\leq \langle \xi\rangle^{2s-4\kappa}.$$
		
		\item Supposing  $|\xi_2|\leq \frac{2}{3}|\xi_1|$, we have,  $$\langle \xi\rangle^{2s}\langle \xi_1\rangle^{-2\kappa} \langle \xi_2\rangle^{-2\kappa}\leq \langle \xi\rangle^{2s-4\kappa}.$$
		
		\item The last case, $\frac{2}{3}|\xi_2|<|\xi_1|<\frac{3}{2}|\xi_2|$.
		\begin{enumerate}
			\item If $\xi_1,\ \xi_2\geq 0$ then $ \frac{2}{3}\xi_2<\xi_1<\frac{3}{2}\xi_2\Longrightarrow \frac{5}{3}\xi_2<\xi <\frac{5}{2}\xi_2$. Hence, $$\langle \xi\rangle^{2s}\langle \xi_1\rangle^{-2{\kappa} } \langle \xi_2\rangle^{-2{\kappa} }\leq \langle \xi\rangle^{2s-4\kappa}.$$
			\item If $\xi_1,\ \xi_2\leq 0$ then $ \frac{-2}{3}\xi_2<-\xi_1<\frac{-3}{2}\xi_2\Longrightarrow \frac{-5}{3}\xi_2<-\xi <\frac{-5}{2}\xi_2$, thus $|\xi_2|<\frac{3}{5}|\xi|$. Hence, $$\langle \xi\rangle^{2s}\langle \xi_1\rangle^{-2{\kappa} } \langle \xi_2\rangle^{-2{\kappa} }\leq \langle \xi\rangle^{2s-4\kappa}.$$
			\item If $\xi_1>0$ and $\xi_2<0$ then $\frac{-2}{3}\xi_2<\xi_1<\frac{-3}{2}\xi_2\Longrightarrow \frac{1}{3}\xi_2<\xi <\frac{-1}{2}\xi_2\Longrightarrow |\xi|<\frac{1}{2}|\xi_2|$.
			\item If $\xi_1<0$ and $\xi_2>0$ then $\frac{2}{3}\xi_2<-\xi_1<\frac{3}{2}\xi_2\Longrightarrow \frac{-1}{3}\xi_2<-\xi <\frac{1}{2}\xi_2$, consequently $|\xi|<\frac{1}{2}|\xi_2|$.
		\end{enumerate}
		
	\end{enumerate}
	The cases (1), (2), (3.a) and (3.b) are true  for $\kappa<0 $ and $s<2{\kappa}+1$.
	
	Indeed, given $\mathcal{A}\subset \mathbb{R}^4$ the set of the elements of $\mathbb{R}^4$ that satisfies one of the conditions (1), (2), (3.a) or (3.b), given $\mathcal{B}=\mathbb{R}^4 \setminus \mathcal{A}$.
	Now consider 
	$\mathcal{A}_{i}=\mathcal{S}_{i}\cap \mathcal{A}$ and $\mathcal{B}_{i}=\mathcal{S}_{i}\cap \mathcal{B}.$

	Analyzing the restrictions $\mathcal{A}_{i}$, we get:
	\begin{eqnarray*}
		J_4 &&= \dfrac{1}{\langle\lambda\rangle^{2d}}\displaystyle\int_{\mathbb{R}}\dfrac{\langle \xi\rangle^{2s}\langle \xi_1\rangle^{-2{\kappa} } \langle \xi_2\rangle^{ -2{\kappa} }\chi_{\mathcal{B}_{1}}}{\langle\tau+\xi_2^2-2\xi\xi_2+\xi^2\rangle^{2b}}d\xi_2\\
		&&\leq \displaystyle\int_{\mathbb{R}}\dfrac{\langle \xi\rangle^{2s-4\kappa-4d} \chi_{\mathcal{A}_{1}}}{\langle\tau+\xi_2^2-2\xi\xi_2+\xi^2\rangle^{2b}}d\xi_2.
	\end{eqnarray*}
	Then $J_4$ is bounded for $s\leq 2{\kappa}+2d$ and $b<3/4$.
	\begin{eqnarray*}
		J_5 &&= \dfrac{1}{\langle\lambda_2\rangle^{2b}}\displaystyle\int_{\mathbb{R}}\dfrac{\langle \xi\rangle^{2s}\langle \xi_1\rangle^{-2{\kappa} } \langle \xi_2\rangle^{-2{\kappa} }\chi_{\mathcal{A}_{2}}}{\langle\tau_2+(a-1)\xi^2-\xi_2^2+2\xi\xi_2\rangle^{2d}}d\xi\\
		&&\leq \displaystyle\int_{\mathbb{R}}\dfrac{\langle \xi\rangle^{2s-4\kappa-4b}\chi_{\mathcal{A}_{2}}}{\langle\tau_2+(a-1)\xi^2-\xi_2^2+2\xi\xi_2\rangle^{2d}}d\xi \ \ 
	\end{eqnarray*}
	and $J_5$ is bounded for $s\leq 2{\kappa}+2b$ and $1/2<b$.
	\begin{eqnarray*}
		J_6 &&= \dfrac{1}{\langle\lambda_1\rangle^{2b}}\displaystyle\int_{\mathbb{R}}\dfrac{\langle \xi\rangle^{2s}\langle \xi_1\rangle^{-2{\kappa} } \langle \xi_2\rangle^{-2{\kappa} }\chi_{\mathcal{A}_{3}}}{\langle\tau_1+a\xi^2+\xi_2^2\rangle^{2d}}d\xi_2\\
		&&\leq \displaystyle\int_{\mathbb{R}}\dfrac{\langle \xi\rangle^{2s-4\kappa-4b}\chi_{\mathcal{A}_{3}}}{\langle\tau_1+a\xi^2+\xi_2^2\rangle^{2d}}d\xi_2.
	\end{eqnarray*}
	Then $J_6$ is also bounded for $s\leq 2{\kappa}+2b$ and $1/2<b$.
	
	To analyze the remaining cases (which is equivalent to supposing $|\xi|<\frac{1}{2}|\xi_2|$ and $|\xi_1|\sim |\xi_2|$) let us consider
	them as regions $\mathcal{B}_{i}$: 
	
	We start by estimating $J_4$.
	\begin{eqnarray*}
		J_4 &&= \dfrac{1}{\langle\lambda\rangle^{2d}}\displaystyle\int_{\mathbb{R}}\dfrac{\langle \xi\rangle^{2s}\langle \xi_1\rangle^{-2{\kappa} } \langle \xi_2\rangle^{-2{\kappa} }\chi_{\mathcal{B}_{1}}}{\langle\tau+\xi_2^2-2\xi\xi_2+\xi^2\rangle^{2b}}d\xi_2\\
		&&\leq \displaystyle\int_{\mathbb{R}}\dfrac{\langle \xi\rangle^{2s-4d}\langle \xi_1\rangle^{-4\kappa} \chi_{\mathcal{B}_{1}}}{\langle\tau+\xi_2^2-2\xi\xi_2+\xi^2\rangle^{2b}}d\xi_2\\
		&& \leq \displaystyle\int_{\mathbb{R}}\dfrac{\langle \xi\rangle^{2s-4d}\langle \xi_1\rangle^{-4\kappa} \chi_{\mathcal{B}_{1}}}{2|\xi_2-\xi|\langle\eta \rangle^{2b}}d\eta
	\end{eqnarray*}
	Now, $|\xi_2-\xi|\geq |\xi_2|-|\xi|\geq \frac{1}{2}|\xi_2|\sim \frac{1}{2}|\xi_1|$.
	
	Hence,
	$J_4\leq \langle \xi\rangle^{2s-4d}\langle \xi_1\rangle^{-4\kappa-1}\displaystyle\int_{\mathbb{R}}\dfrac{d\eta}{\langle\eta \rangle^{2b}}$, that is bounded because $2b>1$ and $2s\leq 4\kappa+2$.
	$$\langle \xi\rangle^{2s-4d}\langle \xi_1\rangle^{-4\kappa-1}\leq \langle \xi \rangle^{2s-4\kappa-4d-1}.$$
	We continue to estimate $J_5$:
	\begin{eqnarray*}
		J_5 &&= \dfrac{1}{\langle\lambda_2\rangle^{2b}}\displaystyle\int_{\mathbb{R}}\dfrac{\langle \xi\rangle^{2s}\langle \xi_1\rangle^{-2{\kappa} } \langle \xi_2\rangle^{-2{\kappa} }\chi_{\mathcal{B}_{2}}}{\langle\tau_2+(a-1)\xi^2-\xi_2^2+2\xi\xi_2\rangle^{2d}}d\xi\\
		&&\leq  \dfrac{1}{\langle\lambda_2\rangle^{2b}}\displaystyle\int_{\mathbb{R}}\dfrac{\langle \xi\rangle^{2s}\langle \xi_2\rangle^{-4\kappa}\chi_{\mathcal{B}_{2}}}{\langle\tau_2+(a-1)\xi^2-\xi_2^2+2\xi\xi_2\rangle^{2d}}d\xi
	\end{eqnarray*}
	Setting $\eta=\tau_2+(a-1)\xi^2-\xi_2^2+2\xi\xi_2$, such that $d\eta=2(\xi_2+(a-1)\xi)d\xi$. Now, as $0<a<\frac{1}{2},$ it follows $|a-1|< 1$ and therefore $|\xi_2+(a-1)\xi|\geq \frac{1}{2}|\xi_2|$.
	Observe still that 
	\begin{eqnarray*}
		|\eta| &&=|\tau_2+(a-1)\xi^2-\xi_2^2+2\xi\xi_2|\\
		&&=|\lambda_2+((a-1)\xi^2-2\xi_2^2+2\xi\xi_2)|\\
		&&\leq |\lambda_2|+|(a-1)\xi^2-2\xi_2^2+2\xi\xi_2|\leq |\tau_2+\xi_2|+4|\xi_2|^2\\
		&& \leq c|\lambda_2|.
	\end{eqnarray*}
	Thus,
	\begin{eqnarray*}
		J_5 && \leq  \dfrac{1}{\langle\lambda_2\rangle^{2b}}\displaystyle\int_{\langle\eta \rangle \leq c\langle\lambda_2\rangle}\dfrac{\langle \xi\rangle^{2s}\langle \xi_2\rangle^{-4\kappa-1}}{\langle \eta\rangle^{2d}}d\eta \\
		&&\leq \dfrac{1}{\langle\lambda_2\rangle^{2b}}\displaystyle\int_{\langle\eta \rangle\leq c\langle\lambda_2\rangle}\dfrac{\langle \xi_2\rangle^{\max\{0,2s\}-4\kappa-1}}{\langle\eta \rangle^{2d}}d\eta, \ \mbox{ because } |\xi|<\frac{1}{2}|\xi_2|\\
		&& \leq \langle \xi_2\rangle^{\max\{0,2s\}-4\kappa-1}\dfrac{\langle\lambda_2\rangle^{2d}}{\langle\lambda_2\rangle^{2b}}\\
		&&\leq  \langle \xi_2\rangle^{\max\{0,2s\}-4\kappa-1}\langle\lambda_2\rangle^{-2b+2d}\leq \langle \xi_2\rangle^{\max\{0,2s\}-4\kappa-1-2b+2d}.
	\end{eqnarray*}
	We prove $J_6$.
	Remember that 
	\begin{eqnarray*}
		J_6 &&= \dfrac{1}{\langle\lambda_1\rangle^{2b}}\displaystyle\int_{\mathbb{R}}\dfrac{\langle \xi\rangle^{2s}\langle \xi_1\rangle^{-2{\kappa} } \langle \xi_2\rangle^{-2{\kappa} }\chi_{\mathcal{B}_{3}}}{\langle\tau_1+a\xi^2+\xi_2^2\rangle^{2d}}d\xi_2\\
		&& \leq \dfrac{1}{\langle\lambda_1\rangle^{2b}}\displaystyle\int_{\mathbb{R}}\dfrac{\langle \xi\rangle^{2s}\langle \xi_1\rangle^{-4\kappa} \chi_{\mathcal{B}_{3}}}{\langle\tau_1+a\xi^2+\xi_2^2\rangle^{2d}}d\xi_2.
	\end{eqnarray*}
	Let $\eta=\tau_1+a\xi^2+\xi_2^2$, such that $d\eta=2\xi_2d\xi_2$. Now, 
	\begin{eqnarray*}
		|\eta| &&=|\tau_1+a\xi^2+\xi_2^2|\\
		&&=|(\lambda_1)+(a\xi^2+\xi_2^2-\xi_1^2)|\\
		&&\leq c |\lambda_1|.
	\end{eqnarray*}
	And using that  $|\xi_1|\sim |\xi_2|$, we obtain
	\begin{eqnarray*}
		J_6 &&\leq \dfrac{1}{\langle\lambda_1\rangle^{2b}}\displaystyle\int_{\langle\eta \rangle\leq c\langle\lambda_1\rangle}\dfrac{\langle \xi\rangle^{2s}\langle \xi_1\rangle^{-4\kappa} }{|\xi_1|\langle\eta \rangle^{2d}}d\xi_2\\
		&&\leq \langle \xi_1\rangle^{\max\{0,2s\}-4\kappa-1}\dfrac{\langle\lambda_1\rangle^{2d}}{\langle\lambda_1\rangle^{2b}}\\
		&& \leq \langle \xi_1\rangle^{\max\{0,2s\}-4\kappa-1-2b+2d}.
	\end{eqnarray*}
	As $1/2<b<3/4$ and $1/4<d<1/2$ we get $-1<-2b+2d<0$ and hence we can take $b$ and $d$ so that $2s-4\kappa-1-2b+2d<0$ if $s<2k+1$.
	
	Then, we completed the proof of Proposition \ref{p2}.
\end{proof}
\begin{remark}
	The lines $s=-{\kappa}-1$ and $s=2{\kappa}+1$ intersect each other at the point where $\kappa=-\frac{2}{3}$.
\end{remark}

\subsection{Bilinear estimates for $\sigma<2$}

For $a>1/2$, we have some results present below.
\begin{prop}\label{p3}
	Assume that $a>1/2$ (equivalently $\sigma <2$), $u\in X^{\kappa,b}$  and $v\in X_a^{s,b}$, then the bilinear estimate below holds if $1/2<b<3/4$, $1/4<d<1/2$ and $|\kappa|-s\leq 1/2$.
	
	\begin{equation}
		\left\|\overline{u}\cdot v \right\|_{X^{\kappa,-d}}\leq C \left\|u \right\|_{X^{\kappa,b}}\cdot \left\|v \right\|_{X^{s,b}}. 
	\end{equation}
\end{prop}

The second estimate tells us that 
\begin{prop}\label{p4}
	Let $a>1/2$ (equivalently $\sigma <2$), $u, \tilde{u}\in X^{\kappa,b}$  with $1/2<b<3/4$ and   $1/4<d<1/2$. The estimate 
	
	\begin{equation}
		\left\|u\cdot \tilde{u} \right\|_{X^{s,-d}}\leq C \left\|u \right\|_{X^{\kappa,b}}\cdot \left\|\tilde{u} \right\|_{X^{\kappa,b}}
	\end{equation}
	holds for $s\leq \min\left\lbrace \kappa +1/2,\  \ 2\kappa +1/2\right\rbrace $.
\end{prop}

\begin{proof}[Proof of Proposition \ref{p3}:]
	
	We start by considering the dispersion relation.
	
	Note that 
	\begin{eqnarray*}
		|\omega-\omega_1-\omega_2|&=& |\xi^2+\xi_1^2-a\xi_2^2|\\
		&\geq & |2\xi^2-2\xi\xi_2+(1-a)\xi_2^2|, \ \ \mbox{ using }\  a> \frac{1}{2} \ \mbox{we have }\\
		&=& 2|\xi-\mu_a \xi_2|\cdot |\xi-(1-\mu_a)\xi_2|, \mbox{ where }\ \mu_a=\frac{1-\sqrt{2a-1}}{2}.
	\end{eqnarray*}
	Note that the above dispersion relation  has two regions: the lines $\xi=\mu_a\xi_2$ and $\xi=(1-\mu_a)\xi_2$ making it difficult to use the relationship. Observe that if $a=\frac{1}{2}$ then $\mu_a=1-\mu_a=\frac{1}{2}$ and if $a=1$ then $\mu_a=0$ (the case $a=\frac{1}{2}$ 
	will be treated separately, while
	the case $a=1$ does not require much attention despite being the case without modification) \cite{pava-2007}.
	
	Before doing it, consider
	\begin{eqnarray*}
		\mathcal{A}_1&=& \{ |\xi_2| \leq 1\}\subset \mathbb{R}^4,  \\
		\mathcal{A}_2&=& \left\{ |\xi_2|\geq 1, \left|(1-a)\xi_2-\xi\right|>\frac{2a-1}{4}|\xi_2|\right\}\subset \mathbb{R}^4,  \\
		\mathcal{A}_3&=& \left\{ |\xi_2|\geq 1, \left|\xi-\frac{1}{2}\xi_2\right|>\frac{2a-1}{4}|\xi_2|\right\}\subset \mathbb{R}^4. 
	\end{eqnarray*}
	Note that if $\left|\xi-\frac{1}{2}\xi_2\right|\leq\frac{2a-1}{4}|\xi_2|$, $ \left|\xi-\frac{1}{2}\xi_2\right|\leq\frac{2a-1}{4}|\xi_2|$ and  $|\xi_2|\geq 1$ then
	\begin{eqnarray*}
		\left(a-\frac{1}{2}\right)|\xi_2|&=&\left| \left(\xi-\frac{1}{2}\xi_2\right)+\left((1-a)\xi_2-\xi\right)\right|\\
		&\leq & \frac{2a-1}{4}|\xi_2|+\frac{2a-1}{4}|\xi_2|=\frac{1}{2}\left(a-\frac{1}{2}\right)|\xi_2|.
	\end{eqnarray*}
	This contradiction implies  $\mathbb{R}^4= \mathcal{A}_1\cup \mathcal{A}_2\cup \mathcal{A}_3$.\\
	Now consider,
	\begin{eqnarray*}
		\mathcal{A}_{3,1}&=& \mathcal{A}_3 \cap \left\{ |\omega|\geq \max \{|\omega_1|, |\omega_2|\} \right\}, \\
		\mathcal{A}_{3,2}&=& \mathcal{A}_3 \cap \left\{ |\omega_2|\geq \max \{|\omega_1|, |\omega|\}\right\},\\
		\mathcal{A}_{3,3}&=& \mathcal{A}_3\cap \left\{ |\omega_1|\geq \max \{|\omega|, |\omega_2|\}\right\}.
	\end{eqnarray*}
	Remember that $|2\xi^2+\xi_2^2-2\xi\xi_2|\leq 3 \max \{|\omega|,|\omega_1|,|\omega_2|\}.$
	
	Now, we define the regions $\mathcal{R}_i$ (analogously to the proof of proposition (\ref{p1})). 
	Let $\mathcal{R}_1=\mathcal{A}_1\cup\mathcal{A}_2\cup \mathcal{A}_{3,1}$, $\mathcal{R}_2=\mathcal{A}_{3,2}$ and $\mathcal{R}_3=\mathcal{A}_{3,3}.$
	
	We will show that $J_1$ is bounded. Indeed, if $|\xi_2|\leq 1$ then $J_1$ is equivalent to $$\dfrac{1}{\langle\omega\rangle^{2d}}\displaystyle\int_{|\xi_2|\leq 1}\dfrac{1}{\langle\tau-(a-1)\xi_2^2-2\xi\xi_2+\xi^2\rangle^{2b}}d\xi_2\leq c.$$ If $|\xi_2|\geq 1$ then  
	$$J_1 \leq \dfrac{1}{\langle\omega\rangle^{2d}}\displaystyle\int_{|\xi_2|\geq 1}\dfrac{\langle\xi_2\rangle^{-2s+2|{\kappa}|}\chi_{\mathcal{A}_2}}{\langle\tau-(a-1)\xi_2^2-2\xi\xi_2+\xi^2\rangle^{2b}}d\xi_2.$$ 
	Changing the variable $\eta=\tau-(a-1)\xi_2^2-2\xi\xi_2+\xi^2$, we get $$d\eta=-2\left((1-a)\xi_2-\xi\right)d\xi_2$$ and, using the fact that $|{\kappa}|-s\leq 1/2$, obtain the following 
	\begin{eqnarray*}
		\dfrac{1}{\langle\omega\rangle^{2d}}\displaystyle\int_{|\xi_2|\geq 1}\dfrac{\langle\xi_2\rangle^{-2s+2|{\kappa}|}\chi_{\mathcal{A}_2}}{\langle\tau-(a-1)\xi_2^2-2\xi\xi_2+\xi^2\rangle^{2b}}d\xi_2 &\leq & c\dfrac{1}{\langle\omega\rangle^{2d}}\displaystyle\int_{|\xi_2|\geq 1}\dfrac{\langle\xi_2\rangle^{-2s+2|{\kappa}|-1}\chi_{\mathcal{A}_2}}{\langle\eta \rangle^{2b}}d\eta\\
		&\leq & c\dfrac{1}{\langle\omega\rangle^{2d}}\displaystyle\int_{\mathbb{R}}\dfrac{1}{\langle\eta \rangle^{2b}}d\eta\leq c.
	\end{eqnarray*}
	Now, note that in $\mathcal{A}_{3,1}$ we have:
	$$
	\left|(1-a)\xi_2-\xi\right|=\left|\frac{1}{2}\xi_2-\xi +\left(\frac{1}{2}-a\right)\xi_2\right|\geq \left(a-\frac{1}{2}\right)\left|\xi_2\right|-\frac{2a-1}{4}\left|\xi_2\right|\geq c |\xi_2|.
	$$
	To complete the estimate of $J_1$ we change variable to get:
	\begin{eqnarray*}
		\dfrac{1}{\langle\omega\rangle^{2d}}\displaystyle\int_{|\xi_2|\geq 1}\dfrac{\langle\xi_2\rangle^{-2s+2|{\kappa}|}\chi_{\mathcal{A}_{3,1}}}{\langle\tau-(a-1)\xi_2^2-2\xi\xi_2+\xi^2\rangle^{2b}}d\xi_2 &\leq & c\dfrac{1}{\langle\omega\rangle^{2d}}\displaystyle\int_{|\xi_2|\geq 1}\dfrac{\langle\xi_2\rangle^{-2s+2|{\kappa}|-1}\chi_{\mathcal{A}_{3,1}}}{\langle\eta \rangle^{2b}}d\eta\\
		&\leq & c\dfrac{1}{\langle\omega\rangle^{2d}}\displaystyle\int_{\mathbb{R}}\dfrac{1}{\langle\eta \rangle^{2b}}d\eta\leq c.
	\end{eqnarray*}
	To prove that $J_2$ is bounded just observe that 
	\begin{eqnarray*}
		\dfrac{1}{\langle\tau+a\xi^2\rangle^{2b}}\displaystyle\int_{\mathbb{R}}\dfrac{\langle\xi_2\rangle^{-2s+2|{\kappa}|}\chi_{\mathcal{R}_2}}{\langle\tau_2+2\xi^2+\xi_2^2-2\xi\xi_2\rangle^{2d}}d\xi &= &  \dfrac{1}{\langle\omega_2 \rangle^{2b}}\displaystyle\int_{\mathbb{R}}\dfrac{\langle\xi_2\rangle^{-2s+2|{\kappa}|}\chi_{\mathcal{A}_{3,2}}}{\langle\tau_2+2\xi^2+\xi_2^2-2\xi\xi_2\rangle^{2d}}d\xi\\
		&\leq & \dfrac{1}{\langle\omega_2 \rangle^{2b}}\displaystyle\int_{\langle\eta\rangle\leq 4 \langle\omega_2\rangle}\dfrac{\langle\xi_2\rangle^{-2s+2|{\kappa}|-1}}{\langle\eta\rangle^{2d}}d\eta\\
		&\leq & \dfrac{1}{\langle\omega_2 \rangle^{2b-2d}}.
	\end{eqnarray*}
	In the first inequality above, we made the change of variable $\eta=\tau_2+2\xi^2+\xi_2^2-2\xi\xi_2$ and used the fact that 
	$$
	|\eta|=|\omega_2+ (\omega-\omega_1-\omega_2)|\leq 4 |\omega_2|.
	$$
	We estimate $J_3$. 
	Analogously to the last estimate we get 
	\begin{eqnarray*}
		\dfrac{1}{\langle\tau-\xi_1^2\rangle^{2b}}\displaystyle\int_{\mathbb{R}}\dfrac{\langle\xi_2\rangle^{-2s+2|{\kappa}|}\chi_{\mathcal{R}_3}}{\langle\tau_1-a\xi_2^2+\xi^2\rangle^{2d}}d\xi_2 &= &\dfrac{1}{\langle\omega_1\rangle^{2b}}\displaystyle\int_{|\xi_2|>1}\dfrac{\langle\xi_2\rangle^{-2s+2|{\kappa}|}\chi_{\mathcal{A}_{3,3}}}{\langle\tau_1-a\xi_2^2+\xi^2\rangle^{2d}}d\xi_2 \\
		&\leq & \dfrac{1}{\langle\omega_1\rangle^{2b}}\displaystyle\int_{\langle\eta\rangle\leq 4 \langle\omega_1\rangle}\dfrac{\langle\xi_2\rangle^{-2s+2|{\kappa}|-1}}{\langle\eta\rangle^{2d}}d\eta\\
		&\leq & \dfrac{1}{\langle\omega_1 \rangle^{2b-2d}}.
	\end{eqnarray*}
	Note that we used the fact that
	$\tau_1-a\xi_2^2+\xi^2=\omega_1+(\omega -\omega_1-\omega_2)$.
	
	This finishes the proof of the first inequality.
\end{proof}
\begin{proof}[Proof of Proposition \ref{p4}]
Initially, we have that:
	\begin{eqnarray*}
		|\lambda-\lambda_1-\lambda_2|&=& |a\xi^2-\xi_1^2-\xi_2^2|\\
		&\geq & |2\xi_2^2-2\xi\xi_2+(1-a)\xi^2| \ \ \mbox{ using }\  a> \frac{1}{2} \ \mbox{ we have }\\
		&=& 2|\xi_2-\mu_a \xi|\cdot |\xi_2-(1-\mu_a)\xi|, \mbox{ where }\ \mu_a=\frac{1-\sqrt{2a-1}}{2}.
	\end{eqnarray*}
	The dispersion relation above is zero in two straight lines.
	
	Now, we define, 
	\begin{eqnarray*}
		\mathcal{B}_1&=& \{ |\xi| \leq 1\}\subset \mathbb{R}^4,  \\
		\mathcal{B}_2&=& \left\{ |\xi|\geq 1, \left|\xi_2-\frac{1}{2}\xi\right|>\frac{2a-1}{4}|\xi|\right\}\subset \mathbb{R}^4, \\
		\mathcal{B}_3&=& \left\{ |\xi|\geq 1, \left|(1-a)\xi-\xi_2\right|>\frac{2a-1}{4}|\xi|\right\}\subset \mathbb{R}^4.  
	\end{eqnarray*}
	Note that if $\left|\xi_2-\frac{1}{2}\xi\right|\leq\frac{2a-1}{4}|\xi|$ and $ \left|\xi_2-\frac{1}{2}\xi\right|\leq\frac{2a-1}{4}|\xi|$ and still $|\xi|> 1$ then
	\begin{eqnarray*}
		\left(a-\frac{1}{2}\right)|\xi|&=&\left| \left(\xi_2-\frac{1}{2}\xi\right)+\left((1-a)\xi-\xi_2\right)\right|\\
		&\leq & \frac{2a-1}{4}|\xi|+\frac{2a-1}{4}|\xi|=\frac{1}{2}\left(a-\frac{1}{2}\right)|\xi|.
	\end{eqnarray*}
	Again, this contradiction implies $\mathbb{R}^4= \mathcal{B}_1\cup \mathcal{B}_2\cup \mathcal{B}_3.$
	
	Now, consider
	\begin{eqnarray*}
		\mathcal{B}_{3,1}&=& \mathcal{B}_3 \cap \left\{ |\lambda|\geq \max \{|\lambda_1|, |\lambda_2|\} \right\}, \\
		\mathcal{B}_{3,2}&=& \mathcal{B}_3 \cap \left\{ |\lambda_2|\geq \max \{|\lambda_1|, |\lambda|\}\right\},\\
		\mathcal{B}_{3,3}&=& \mathcal{B}_3\cap \left\{ |\lambda_1|\geq \max \{|\lambda|, |\lambda_2|\}\right\}.
	\end{eqnarray*}
	We define the regions $\mathcal{S}_i$ (analogously to the proof of proposition (\ref{p2})), 
	setting $\mathcal{S}_1=\mathcal{B}_1\cup\mathcal{B}_2\cup \mathcal{B}_{3,1}$, $\mathcal{S}_2=\mathcal{B}_{3,2}$ and $\mathcal{S}_3=\mathcal{B}_{3,3}.$
	
	For $\kappa\geq 0$, we have
	$\langle \xi_1 \rangle^{-2{\kappa} } \langle \xi_2 \rangle^{-2{\kappa} }\leq \langle \xi \rangle^{-2{\kappa} }$:
	
	\begin{equation}\label{j41}
		J_4 \leq \dfrac{1}{\langle\lambda\rangle^{2d}}\displaystyle\int_{\mathbb{R}}\dfrac{\langle \xi\rangle^{2s-2{\kappa} } \chi_{\mathcal{S}_1}}{\langle\tau+\xi_2^2-2\xi\xi_2+\xi^2\rangle^{2b}}d\xi_2,
	\end{equation}
	
	\begin{equation}\label{j51}
		J_5 \leq \dfrac{1}{\langle\lambda_2\rangle^{2b}}\displaystyle\int_{\mathbb{R}}\dfrac{\langle \xi\rangle^{2s-2{\kappa} }\chi_{\mathcal{S}_2}}{\langle\tau_2+(a-1)\xi^2-\xi_2^2+2\xi\xi_2\rangle^{2d}}d\xi, 
	\end{equation}
	
	\begin{equation}\label{j61}
		J_6 \leq \dfrac{1}{\langle\lambda_1\rangle^{2b}}\displaystyle\int_{\mathbb{R}}\dfrac{\langle \xi\rangle^{2s-2{\kappa} }\chi_{\mathcal{S}_3}}{\langle\tau_1+a\xi^2+\xi_2^2\rangle^{2d}}d\xi_2. 
	\end{equation}
	To complete the proof that $J_4$ is bounded it is suffices to show that  (\ref{j41}) satisfies: 
	\begin{equation*}
		\dfrac{1}{\langle\lambda\rangle^{2d}}\displaystyle\int_{\mathbb{R}}\dfrac{\langle \xi\rangle^{2s-2{\kappa} } \chi_{\mathcal{B}_1}}{\langle\tau+\xi_2^2-2\xi\xi_2+\xi^2\rangle^{2b}}d\xi_2\leq \dfrac{1}{\langle\lambda\rangle^{2d}}\displaystyle\int_{\mathbb{R}}\dfrac{1}{\langle\tau+\xi_2^2-2\xi\xi_2+\xi^2\rangle^{2b}}d\xi_2\leq c;
	\end{equation*}
	
	\begin{equation*} 
		\dfrac{1}{\langle\lambda\rangle^{2d}}\displaystyle\int_{\mathbb{R}}\dfrac{\langle \xi\rangle^{2s-2{\kappa} } \chi_{\mathcal{B}_2}}{\langle\tau+\xi_2^2-2\xi\xi_2+\xi^2\rangle^{2b}}d\xi_2\leq \dfrac{1}{\langle\lambda\rangle^{2d}}\displaystyle\int_{\mathbb{R}}\dfrac{\langle \xi\rangle^{2s-2{\kappa}-1}}{\langle\eta\rangle^{2b}}d\xi_2\leq c \mbox{ and }
	\end{equation*}
	
	\begin{equation*}
		\dfrac{1}{\langle\lambda\rangle^{2d}}\displaystyle\int_{\mathbb{R}}\dfrac{\langle \xi\rangle^{2s-2{\kappa} } \chi_{\mathcal{B}_{3,1}}}{\langle\tau+\xi_2^2-2\xi\xi_2+\xi^2\rangle^{2b}}d\xi_2\leq \dfrac{1}{\langle\lambda\rangle^{2d}}\displaystyle\int_{\mathbb{R}}\dfrac{\langle \xi\rangle^{2s-2{\kappa}-1}}{\langle\eta\rangle^{2b}}d\xi_2\leq c.
	\end{equation*}
	In the estimates above, we used the fact $b>1/2$ and also the fact that 
	\begin{eqnarray*}
		\left|\xi_2-\frac{1}{2}\xi\right| &= & \left|(1-a)\xi-\xi_2+ \left(a-\frac{1}{2}\right)\xi\right|\\
		&\geq & \left(a-\frac{1}{2}\right)|\xi| -\left|(1-a)\xi-\xi_2\right|\\
		& \geq & \left(a-\frac{1}{2}\right)|\xi|-\frac{1}{2}\left(a-\frac{1}{2}\right)|\xi|= \frac{2a-1}{4}|\xi|.
	\end{eqnarray*}
	Let us estimate (\ref{j51}), using the fact that $$\eta=\tau_2+(a-1)\xi^2-\xi_2^2+2\xi\xi_2=\lambda_2+(\lambda-\lambda_1-\lambda_2),$$ which give us $d\eta=2((1-a)\xi-\xi_2)d\xi$, so
	\begin{eqnarray*}
		\dfrac{1}{\langle\lambda_2\rangle^{2b}}\displaystyle\int_{\mathbb{R}}\dfrac{\langle \xi\rangle^{2s-2{\kappa} }\chi_{\mathcal{B}_{3,2}}}{\langle\tau_2+(a-1)\xi^2-\xi_2^2+2\xi\xi_2\rangle^{2d}}d\xi & \leq & \dfrac{1}{\langle\lambda_2\rangle^{2b}}\displaystyle\int_{\langle\eta\rangle\leq 4\langle_2\lambda\rangle}\dfrac{\langle \xi\rangle^{2s-2{\kappa}-1}}{\langle\eta\rangle^{2d}}d\eta \\
		& \leq & \dfrac{1}{\langle\lambda_2\rangle^{2b-2d}}\leq c.
	\end{eqnarray*}
	Now let us estimate (\ref{j61}). This is completely analogous to the previous estimate.
	\begin{eqnarray*}
		\dfrac{1}{\langle\lambda_1\rangle^{2b}}\displaystyle\int_{\mathbb{R}}\dfrac{\langle \xi\rangle^{2s-2{\kappa} }\chi_{\mathcal{S}_3}}{\langle\tau_1+a\xi^2+\xi_2^2\rangle^{2d}}d\xi_2 & \leq & \dfrac{1}{\langle\lambda_1\rangle^{2b}}\displaystyle\int_{\langle\eta\rangle\leq 4\langle\lambda_1\rangle}\dfrac{\langle \xi\rangle^{2s-2{\kappa}-1}}{\langle\eta\rangle^{2d}}d\eta \\
		& \leq & \dfrac{1}{\langle\lambda_1\rangle^{2b-2d}}\leq c.
	\end{eqnarray*}
	This concludes the case $\kappa \geq 0$.
	
	\vspace{0.5cm}
	The case $\kappa<0$ will be separated into sub-cases: 
	
	\begin{enumerate}
		\item Supposing $|\xi_1|\leq \frac{2}{3}|\xi_2|$, then, $\langle\xi_1\rangle^{-2{\kappa} }\langle\xi_2\rangle^{-2{\kappa} }\leq \langle \xi_2\rangle^{-4\kappa}$. Moreover,$|\xi_2|\leq |\xi_1|+|\xi|\leq \frac{2|\xi_2|}{3}+|\xi|$, hence $|\xi_2|\leq 3|\xi|$. Therefore, $$\langle \xi\rangle^{2s}\langle \xi_1\rangle^{-2{\kappa} } \langle \xi_2\rangle^{-2{\kappa} }\leq \langle \xi\rangle^{2s-4\kappa}.$$
		
		\item Supposing  $|\xi_2|\leq \frac{2}{3}|\xi_1|$ we have the same result, that is,  $$\langle \xi\rangle^{2s}\langle \xi_1\rangle^{-2{\kappa} } \langle \xi_2\rangle^{-2{\kappa} }\leq \langle \xi\rangle^{2s-4\kappa}.$$
		
		\item For the case, $\frac{2}{3}|\xi_2|<|\xi_1|<\frac{3}{2}|\xi_2|$, we need to do the following:
		\begin{enumerate}
			\item If $\xi_1,\ \xi_2\geq 0$ then $ \frac{2}{3}\xi_2<\xi_1<\frac{3}{2}\xi_2\Longrightarrow \frac{5}{3}\xi_2<\xi <\frac{5}{2}\xi_2$. Hence, $$\langle \xi\rangle^{2s}\langle \xi_1\rangle^{-2{\kappa} } \langle \xi_2\rangle^{-2{\kappa} }\leq \langle \xi\rangle^{2s-4\kappa}.$$
			\item If $\xi_1,\ \xi_2\leq 0$ then $ \frac{-2}{3}\xi_2<-\xi_1<\frac{-3}{2}\xi_2\Longrightarrow \frac{-5}{3}\xi_2<-\xi <\frac{-5}{2}\xi_2$, so $|\xi_2|<\frac{3}{5}|\xi|$. Hence, $$\langle \xi\rangle^{2s}\langle \xi_1\rangle^{-2{\kappa} } \langle \xi_2\rangle^{-2{\kappa} }\leq \langle \xi\rangle^{2s-4\kappa}.$$
			\item If $\xi_1>0$ and $\xi_2<0$ then $\frac{-2}{3}\xi_2<\xi_1<\frac{-3}{2}\xi_2\Longrightarrow \frac{1}{3}\xi_2<\xi <\frac{-1}{2}\xi_2$, now $|\xi|<\frac{1}{2}|\xi_2|$.
			\item If $\xi_1<0$ and $\xi_2>0$ then $\frac{2}{3}\xi_2<-\xi_1<\frac{3}{2}\xi_2\Longrightarrow \frac{-1}{3}\xi_2<-\xi <\frac{1}{2}\xi_2$, which give us $|\xi|<\frac{1}{2}|\xi_2|$.
		\end{enumerate}
		
	\end{enumerate}
	
	The cases (1), (2), (3.a) and (3.b) are valid for $\kappa<0 $ and $s<2{\kappa}+\frac{1}{2}$.
	
	Indeed, let $\mathcal{C}\subset \mathbb{R}^4$ be the set of element $\mathbb{R}^4$ that satisfies one of the conditions (1), (2), (3.a) or (3.b).
	Now consider 
	$\mathcal{C}_{i}=\mathcal{S}_{i}\cap \mathcal{C}$.

	Analyzing the restrictions on $\mathcal{C}_{i}$, we get:
	\begin{eqnarray*}
		\dfrac{1}{\langle\lambda\rangle^{2d}}\displaystyle\int_{\mathbb{R}}\dfrac{\langle \xi\rangle^{2s}\langle \xi_1\rangle^{-2{\kappa} } \langle \xi_2\rangle^{-2{\kappa} }\chi_{\mathcal{C}_{1}}}{\langle\tau+\xi_2^2-2\xi\xi_2+\xi^2\rangle^{2b}}d\xi_2 & \leq & \dfrac{1}{\langle\lambda\rangle^{2d}}\displaystyle\int_{\mathbb{R}}\dfrac{\langle \xi\rangle^{2s-4\kappa} \chi_{\mathcal{C}_{1}}}{\langle\tau+\xi_2^2-2\xi\xi_2+\xi^2\rangle^{2b}}d\xi_2\\
		& \leq &\dfrac{1}{\langle\lambda\rangle^{2d}}\displaystyle\int_{\mathbb{R}}\dfrac{\langle \xi\rangle^{2s-4\kappa-1} }{\langle\eta\rangle^{2b}}d\xi_2\\
		& \leq & c, \ \mbox{ because } 1/2<b<1 \mbox{ and } s<2{\kappa}+1/2.
	\end{eqnarray*}

	\begin{eqnarray*}
		\dfrac{1}{\langle\lambda_2\rangle^{2b}}\displaystyle\int_{\mathbb{R}}\dfrac{\langle \xi\rangle^{2s}\langle \xi_1\rangle^{-2{\kappa} } \langle \xi_2\rangle^{-2{\kappa} }\chi_{\mathcal{C}_{2}}}{\langle\tau_2+(a-1)\xi^2-\xi_2^2+2\xi\xi_2\rangle^{2d}}d\xi
		&&\leq \dfrac{1}{\langle\lambda_2\rangle^{2b}}\displaystyle\int_{\langle\eta\rangle\leq 4\langle\lambda_2\rangle}\dfrac{\langle \xi\rangle^{2s-4\kappa-1}}{\langle\eta\rangle^{2d}}d\eta \\
		&&\leq \dfrac{1}{\langle\lambda_2\rangle^{2b-2d}}\leq c.
	\end{eqnarray*}

	\begin{eqnarray*}
		\dfrac{1}{\langle\lambda_1\rangle^{2b}}\displaystyle\int_{\mathbb{R}}\dfrac{\langle \xi\rangle^{2s}\langle \xi_1\rangle^{-2{\kappa} } \langle \xi_2\rangle^{-2{\kappa} }\chi_{\mathcal{C}_{3}}}{\langle\tau_1+a\xi^2+\xi_2^2\rangle^{2d}}d\xi_2
		& &\leq \dfrac{1}{\langle\lambda_1\rangle^{2b}}\displaystyle\int_{\langle\eta\rangle\leq 4\langle\lambda_1\rangle}\dfrac{\langle \xi\rangle^{2s-4\kappa-1}}{\langle\eta\rangle^{2d}}d\eta \\
		&&\leq \dfrac{1}{\langle\lambda_1\rangle^{2b-2d}}\leq c.
	\end{eqnarray*}
	
	Consider $\mathcal{D}=\mathbb{R}^4 \setminus \mathcal{C}$  and $\mathcal{D}_{i}=\mathcal{S}_{i}\cap \mathcal{D}.$
	To obtain the other cases (which is equivalent to supposing $|\xi|<\frac{1}{2}|\xi_2|$ and $|\xi_1|\sim |\xi_2|$) let us consider the regions $\mathcal{D}_{i}$: 
	
	We begin by estimating $J_4$.
	
	\begin{eqnarray*}
		\dfrac{1}{\langle\lambda\rangle^{2d}}\displaystyle\int_{\mathbb{R}}\dfrac{\langle \xi\rangle^{2s}\langle \xi_1\rangle^{-2{\kappa} } \langle \xi_2\rangle^{-2{\kappa} }\chi_{\mathcal{D}_{1}}}{\langle\tau+\xi_2^2-2\xi\xi_2+\xi^2\rangle^{2b}}d\xi_2 & \leq  & \dfrac{1}{\langle\lambda\rangle^{2d}} \displaystyle\int_{\mathbb{R}}\dfrac{\langle \xi\rangle^{2s}\langle \xi_1\rangle^{-4\kappa} \chi_{\mathcal{D}_{1}}}{\langle\tau+\xi_2^2-2\xi\xi_2+\xi^2\rangle^{2b}}d\xi_2\\
		& \leq & \dfrac{1}{\langle\lambda\rangle^{2d}}\displaystyle\int_{\mathbb{R}}\dfrac{\langle \xi\rangle^{2s}\langle \xi_1\rangle^{-2{\kappa}} }{\langle\eta \rangle^{2b}}d\eta.
	\end{eqnarray*}
	Now, $|\xi_2-\xi|\geq |\xi_2|-|\xi|\geq \frac{1}{2}|\xi_2|\sim \frac{1}{2}|\xi_1|$.\\
	Hence,
	$J_4\leq \langle \xi\rangle^{2s-4d}\langle \xi_1\rangle^{-4\kappa-1}\displaystyle\int_{\mathbb{R}}\dfrac{d\eta}{\langle\eta \rangle^{2b}}$, the right hand side is bounded because $2b>1$, $2s\leq 4\kappa+2$ and $1/4<d<1/2$ in addition,
	$$\langle \xi\rangle^{2s-4d}\langle \xi_1\rangle^{-4\kappa-1}\leq \langle \xi \rangle^{2s-4\kappa-1-4d}\leq \langle \xi \rangle^{1-4d}.$$
	Estimating $J_5$:
	\begin{eqnarray*}
		J_5 &&= \dfrac{1}{\langle\lambda_2\rangle^{2b}}\displaystyle\int_{\mathbb{R}}\dfrac{\langle \xi\rangle^{2s}\langle \xi_1\rangle^{-2{\kappa} } \langle \xi_2\rangle^{-2{\kappa} }\chi_{\mathcal{B}_{2}}}{\langle\tau_2+(a-1)\xi^2-\xi_2^2+2\xi\xi_2\rangle^{2d}}d\xi\\
		&&\leq  \dfrac{1}{\langle\lambda_2\rangle^{2b}}\displaystyle\int_{\mathbb{R}}\dfrac{\langle \xi\rangle^{2s}\langle \xi_2\rangle^{-4\kappa}\chi_{\mathcal{B}_{2}}}{\langle\tau_2+(a-1)\xi^2-\xi_2^2+2\xi\xi_2\rangle^{2d}}d\xi.
	\end{eqnarray*}
	Setting $\eta=\tau_2+(a-1)\xi^2-\xi_2^2+2\xi\xi_2$, which give $d\eta=2(\xi_2+(a-1)\xi)d\xi$. As $0<a<\frac{1}{2},$ we have $|a-1|\leq 1$ and therefore $|\xi_2+(a-1)\xi|\geq \frac{1}{2}|\xi_2|$.
	Also we note that 
	\begin{eqnarray*}
		|\eta| &&=|\tau_2+(a-1)\xi^2-\xi_2^2+2\xi\xi_2|\\
		&&=|(\lambda_2)+((a-1)\xi^2-2\xi_2^2+2\xi\xi_2)|\\
		&&\leq |\lambda_2|+|(a-1)\xi^2-2\xi_2^2+2\xi\xi_2|\leq |\tau_2+\xi_2|+4|\xi_2|^2\\
		&& \leq c|\lambda_2|.
	\end{eqnarray*}
	Hence,
	\begin{eqnarray*}
		J_5 && \leq  \dfrac{1}{\langle\lambda_2\rangle^{2b}}\displaystyle\int_{\langle\eta \rangle \leq c\langle\lambda_2\rangle}\dfrac{\langle \xi\rangle^{2s}\langle \xi_2\rangle^{-4\kappa-1}}{\langle \eta\rangle^{2d}}d\eta \\
		&&\leq \dfrac{1}{\langle\lambda_2\rangle^{2b}}\displaystyle\int_{\langle\eta \rangle\leq c\langle\lambda_2\rangle}\dfrac{\langle \xi_2\rangle^{\max\{0,2s\}-4\kappa-1}}{\langle\eta \rangle^{2d}}d\eta \ \ \mbox{ because } |\xi|<\frac{1}{2}|\xi_2|\\
		&& \leq \langle \xi_2\rangle^{\max\{0,2s\}-4\kappa-1}\dfrac{\langle\lambda_2\rangle^{1-2d}}{\langle\lambda_2\rangle^{2b}}\\
		&&\leq  \langle \xi_2\rangle^{\max\{0,2s\}-4\kappa-1}\langle\lambda_2\rangle^{1-2d-2b}\leq \langle \xi_2\rangle^{\max\{0,2s\}-4\kappa-2}.
	\end{eqnarray*}
	Since $1-2b-2d<-1/2$.

	Now, we estimate $J_6$.
	Remembering that 
	\begin{eqnarray*}
		J_6 &&= \dfrac{1}{\langle\lambda_1\rangle^{2b}}\displaystyle\int_{\mathbb{R}}\dfrac{\langle \xi\rangle^{2s}\langle \xi_1\rangle^{-2{\kappa} } \langle \xi_2\rangle^{-2{\kappa} }\chi_{\mathcal{B}_{3}}}{\langle\tau_1+a\xi^2+\xi_2^2\rangle^{2d}}d\xi_2\\
		&& \leq \dfrac{1}{\langle\lambda_1\rangle^{2b}}\displaystyle\int_{\mathbb{R}}\dfrac{\langle \xi\rangle^{2s}\langle \xi_1\rangle^{-4\kappa} \chi_{\mathcal{B}_{3}}}{\langle\tau_1+a\xi^2+\xi_2^2\rangle^{2d}}d\xi_2.
	\end{eqnarray*}
	Using $\eta=\tau_1+a\xi^2+\xi_2^2$, which give $d\eta=2\xi_2d\xi_2$. Now, 
	\begin{eqnarray*}
		|\eta| &&=|\tau_1+a\xi^2+\xi_2^2|\\
		&&=|(\lambda_1)+(a\xi^2+\xi_2^2-\xi_1^2)|\\
		&&\leq c |\lambda_1|.
	\end{eqnarray*}
	By using the fact that $|\xi_1|\sim |\xi_2|$, we have
	\begin{eqnarray*}
		J_6 && \leq \dfrac{1}{\langle\lambda_1\rangle^{2b}}\displaystyle\int_{\langle\eta \rangle\leq c\langle\lambda_1\rangle}\dfrac{\langle \xi\rangle^{2s}\langle \xi_1\rangle^{-4\kappa} }{|\xi_1|\langle\eta \rangle^{2d}}d\xi_2\\
		&&\leq \langle \xi_1\rangle^{\max\{0,2s\}-4\kappa-1}\dfrac{\langle\lambda_1\rangle^{1-2d}}{\langle\lambda_1\rangle^{2b}}\\
		&& \leq \langle \xi_1\rangle^{\max\{0,2s\}-4\kappa-2}.
	\end{eqnarray*}

	And this finishes the proof of Proposition \ref{p4}.
\end{proof}

\begin{remark}
	The lines $s=-{\kappa}-1/2$ and $s=2{\kappa}+1/2$ intersect each other at the point $\kappa=-\frac{1}{3}$.

\end{remark}

\subsection{Bilinear estimates for $\sigma=2$}
Next we prove a new bilinear estimates for the interaction terms in the case $\sigma=2$

\begin{prop}\label{p5}
	Assume that $a=1/2$ (equivalently $\sigma =2$). If $1/2<b<3/4$, $1/4<d<1/2$ and $|\kappa|\leq s$, then for $u\in X^{\kappa,b}$ and $v\in X_a^{s,b}$, the estimate below 	
	\begin{equation}
		\left\|\overline{u}\cdot v \right\|_{X^{\kappa,-d}}\leq C \left\|u \right\|_{X^{\kappa,b}}\cdot \left\|v \right\|_{X^{s,b}} 
	\end{equation}
	holds
\end{prop}

The second bilinear estimate tells us that 
\begin{prop}\label{p6}
	Let $a=1/2$ (equivalently $\sigma =2$) and  $u, \tilde{u}\in X^{\kappa,b}$, then  
	\begin{equation}
		\left\|u\cdot \tilde{u} \right\|_{X^{s,-d}}\leq C \left\|u \right\|_{X^{\kappa,b}}\cdot \left\|\tilde{u} \right\|_{X^{\kappa,b}}
	\end{equation}
	holds if $1/2<b<3/4$, $1/4<d<1/2$ and $0 \leq s\leq \kappa $.
\end{prop}


\begin{proof}[Proof of Proposition \ref{p5}:]
We begin by noting that
	\begin{eqnarray*}
		|\omega-\omega_1-\omega_2|&=& \left|\xi^2+\xi_1^2-\frac{1}{2}\xi_2^2\right|\\
		&=&\left|2\xi^2+2\xi\xi_2+\frac{1}{2}\xi_2^2\right|\\
		&=&2\left|\xi+\frac{1}{2}\xi_2\right|^2.
	\end{eqnarray*}
	
	In this case, we do not have to take the dispersion relation. Then, we consider 
	$\mathcal{R}_1=\mathbb{R}^4$ and $\mathcal{R}_2=\mathcal{R}_3=\varnothing $. Thus, we only need to prove that $J_1$ is bounded.
	If $|\kappa|\leq s$ then $J_1$  is equivalent to $$\dfrac{1}{\langle\omega\rangle^{2d}}\displaystyle\int_{|\xi_2|\leq 1}\dfrac{1}{\langle\tau-\frac12\xi_2^2-2\xi\xi_2+\xi^2\rangle^{2b}}d\xi_2\leq c,$$ 
	since $b>1/2$ and $d>0$. This finishes the proof of the proposition.
\end{proof}

\begin{proof}[Proof of Proposition \ref{p6}:]
	
	As in the previous case, we cannot take advantage of the dispersion relation. So let us take $\mathcal{S}_1=\mathbb{R}^4$ and $\mathcal{S}_2=\mathcal{S}_3=\varnothing$. Note that it is enough to estimate $J_4$. Initially assume that $\kappa \geq 0$, so we get 
	$\langle \xi_1 \rangle^{-2{\kappa} } \langle \xi_2 \rangle^{-2{\kappa} }\leq \langle \xi \rangle^{-2{\kappa} }$:
	
	\begin{equation*}
		J_4 \leq \dfrac{1}{\langle\lambda\rangle^{2d}}\displaystyle\int_{\mathbb{R}}\dfrac{\langle \xi\rangle^{2s-2{\kappa} } \chi_{\mathcal{S}_1}}{\langle\tau+\xi_2^2-2\xi\xi_2+\xi^2\rangle^{2b}}d\xi_2. 
	\end{equation*}
	Finally, since $s\leq \kappa$, $b>1/2$ and $d>0$, we conclude that $J_4$ is bounded.
\end{proof}

\section{Local existence for low regularity data}

In this section we prove, by using the Banach Fixed Point Theorem, the result of local well-posedness. We only show the case $0<a<1/2$ because the others follow the similar arguments.

Consider the following functional space  where we will get our solution:
\begin{equation}
	\Sigma_{\mu}:= \left\{ (u,v)\in X^{\kappa,\frac{1}{2}+\mu}\times X_a^{s,\frac{1}{2}+\mu}; \left\|u \right\|_{X^{\kappa,\frac{1}{2}+\mu}}\leq M_1, \ 
	\left\|v \right\|_{X_a^{s,\frac{1}{2}+\mu}}\leq M_2 \right\},
\end{equation} 
where $0<\mu\ll 1$  and $M_1, M_2>0$ will be chosen after.

We note that $\Sigma_{\mu}$ is a complete metric space with the standard norm:
\begin{equation}\label{soma}
\left\|(u,v)\right\|_{\Sigma_{\mu}}:=  \left\|u \right\|_{X^{\kappa,\frac{1}{2}+\mu}}  + \left\| v\right\|_{X_a^{s,\frac{1}{2}+\mu}}.
\end{equation}
For $(u,v)\in\Sigma_{\mu}$, we define the maps
\begin{equation}
	\Phi_1(u,v)=\psi_1(t) e^{it\partial^2_x}u_0-i\psi_T(t)\displaystyle\int_0^t e^{i(t-t')\partial^2_x}\left\{ \theta u(t') -\left(\overline{u} \cdot v\right)(t')\right\}dt',
\end{equation}

\begin{equation}
	\Phi_2(u,v)=\psi_1(t) e^{iat\partial^2_x}v_0-i\psi_T(t)\displaystyle\int_0^t e^{ia(t-t')\partial^2_x}\left\{ \alpha v(t') -\frac{a}{2}\left(u^2\right)(t')\right\}dt'.
\end{equation}
We will choose $\mu<\mu(\kappa,s)$ where $d=\frac12-2\mu(\kappa,s)$ and $b=\frac12+\mu(\kappa,s)$ satisfy the conditions of propositions \ref{p1} and \ref{p2}.

According to lemma \ref{l2.1}, with $b'=-d$ and propositions \ref{p1} and \ref{p2}, we have
\begin{eqnarray*}
	\left\| \Phi_1(u,v)\right\|_{X^{\kappa,\frac{1}{2}+\mu}}&&\leq c_0 \left\|u_0 \right\|_{H^{\kappa}}+c_1 T^{\mu}\left(\theta \left\|u \right\|_{X^{\kappa,-\frac12+2\mu}} +\left\|\overline{u}v \right\|_{X^{\kappa,-\frac12+2\mu}} \right)\\
	&& \leq c_0\left\| u_0\right\|_{H^{\kappa}}+c_1 T^{\mu}\left(\theta \left\|u \right\|_{X^{\kappa,\frac12+\mu}}+\left\| u\right\|_{X^{\kappa,\frac12+\mu}}
	\left\| v\right\|_{X_a^{s,\frac12+\mu}}\right)\\
	&& \leq c_0 \left\| u_0\right\|_{H^{\kappa}}+c_1 T^{\mu}\bigg(\theta M_1+M_1M_2\bigg),
\end{eqnarray*}

\begin{eqnarray*}
	\left\| \Phi_2(u,v)\right\|_{X_a^{s,\frac{1}{2}+\mu}}&&\leq c_0 \left\|v_0 \right\|_{H^s}+c_2 T^{\mu}\left(\alpha \left\|v \right\|_{X_a^{s,-\frac12+2\mu}} +\frac{a}{2}\left\|u^2 \right\|_{X^{\kappa,-\frac12+2\mu}} \right)\\
	&& \leq c_0\left\| v_0\right\|_{H^s}+c_2 T^{\mu}\left(\alpha
	\left\|v \right\|_{X_a^{s,\frac12+\mu}}+\frac{a}{2}\left\| u\right\|_{X^{\kappa,\frac12+\mu}}^2
	\right)\\
	&& \leq c_0 \left\| v_0\right\|_{H^s}+c_2 T^{\mu}\bigg(\alpha M_2+\frac{a}{2}M_1^2\bigg).
\end{eqnarray*}
Defining $M_1=2c_0\left\|u_0 \right\|_{H^k}$ and $M_2=2c_0\left\|v_0 \right\|_{H^s}$, we have the following 

\begin{equation*}
	\left\| \Phi_1(u,v)\right\|_{X^{\kappa,\frac{1}{2}+\mu}}\leq \dfrac{M_1}{2}+c_1 T^{\mu}\bigg(\theta M_1+M_1 M_2\bigg)
\end{equation*}
and
\begin{equation*}
	\left\| \Phi_2(u,v)\right\|_{X_a^{s,\frac{1}{2}+\mu}}\leq \dfrac{M_2}{2}+c_2 T^{\mu} \bigg(\alpha M_2+\frac{a}{2}M_1^2\bigg).
\end{equation*}
Then $\left(\Phi_1(u,v),\Phi_2(u,v)\right)\in \Sigma_{\mu}$ for 
\begin{equation}\label{t}
	T^{\mu} \leq \dfrac{1}{2}\min\left\{ \dfrac{1}{c_1(\theta+M_2)},\ \dfrac{M_2}{c_2(\alpha M_2+\frac{a}{2}M_1^2)}\right\}.
\end{equation}
Similarly, we have that
\begin{equation*}
	\left\| \Phi_1(u,v)-\Phi_1(\tilde{u},\tilde{v})\right\|_{X^{\kappa,\frac{1}{2}+\mu}}\leq c_3(M_1, M_2) T^{\mu}\bigg( \left\|u-\tilde{u}\right\|_{X^{\kappa,\frac{1}{2}+\mu}}+\left\|v-\tilde{v}\right\|_{X_a^{s,\frac{1}{2}+\mu}}\bigg), 
\end{equation*}

\begin{equation*}
	\left\| \Phi_2(u,v)-\Phi_2(\tilde{u},\tilde{v}) \right\|_{X_a^{s,\frac{1}{2}+\mu}}\leq c_4(M_1, M_2) T^{\mu}\bigg( \left\|u-\tilde{u}\right\|_{X^{\kappa,\frac{1}{2}+\mu}}+\left\|v-\tilde{v}\right\|_{X_a^{s,\frac{1}{2}+\mu}}\bigg).
\end{equation*}
Now, using (\ref{soma}) and inequalities above, we have
\begin{equation}\label{tt}
	\left\|\bigg(\Phi_1(u,v),\Phi_2(u,v) \bigg)- \bigg(\Phi_1(\tilde{u},\tilde{v}),\Phi_2(\tilde{u},\tilde{v}) \bigg) \right\|_{\Sigma_{\mu}}\leq \frac12 \left\|(u,v)-(\tilde{u}, \tilde{v}) \right\|_{\Sigma_{\mu}}. 
\end{equation}
to
$$T^{\mu}\leq \frac14\min \left\{ \frac{1}{c_3(M_1, M_2)}, \frac{1}{c_4(M_1,M_2)} \right\}.$$
Therefore, the map $\Phi_1\times \Phi_2:\Sigma_{\mu}\longrightarrow \Sigma_{\mu} $ is a contraction, and by the Fixed Point Theorem there is a unique solution to the Cauchy problem for $T$ satisfying (\ref{t}) and (\ref{tt}).

\begin{flushright}
	$\Box$
\end{flushright}

\begin{remark}
	The case $p=q=-1$ can be treated by using the same ideas that in the case $p=q=1$, for any $\sigma>0$.
	
\end{remark}

\begin{remark}
	The case $p=-1$ and $q=1$ or $p=1$ and $q=-1$ (for all $\sigma>0$) is the same in the case $p=q=1$ for $\sigma >2$.
	
\end{remark}
\section{Global Well-Posedness Results}
In this section we will study the global well-posedness for the system (\ref{nls-g}) below:

\begin{equation}\label{nls-g}
	\begin{cases}
		i\partial_t u+p\partial^2_x u -\theta u+ \overline{u} v =0 & \\
		i\sigma\partial_t v+q\partial^2_x v -\alpha v+\frac{1}{2}u^2=0,& t\in [-T,T], \ x\in \mathbb{R},  \\
		u(x,0)=u_0(x), \ \   v(x,0)=v_0(x), &   \ \      (u_0,v_0) \in H^{\kappa}(\mathbb{R})\times H^s(\mathbb{R}), 
	\end{cases}
\end{equation}
where $u$ and $v$ are complex valued functions.

One of the interests in working with system of equations in physics is to obtain stability for certain types of solutions. In this case, it is essential to have global well-posedness results.
%

Starting from the conservation law
%
\begin{equation}
	E(u,v)(t)=\left\|u \right\|^2_{L^2} +2\sigma \left\|v \right\|^2_{L^2},
\end{equation}
it is known that if $u$ and $v$ are solutions of this system with  initial conditions  $(u_0,v_0)\in L^2\times L^2$, then $\forall t\in \mathbb{R}$ we have $E(u,v)(t)=E(u,v)(0)=\left\| u_0\right\|^2_{L^2}+ 2\sigma \left\|v_0 \right\|^2_{L^2}.$

\vspace{0.5cm}

Our main result presented here is theorem \ref{global}.
%

To get the above result we will follow the ideas presented in \cite{tao-2001}, \cite{corcho-2007}, \cite{pecher-2005} and \cite{corcho-2009}.


We note here that we did not explore the second quantity conserved for light  regularities, for example, greater than 1, i.e.,
\begin{equation}
	H(u,v)(t)=p\left\|u_x \right\|_{L^2}^2+q\left\|v_x \right\|_{L^2}^2+\theta\left\|u \right\|_{L^2}^2+\alpha\left\|v \right\|_{L^2}^2-\mbox{Re}\langle u^2,\ \overline{v}\rangle_{L^2}.
\end{equation}

\subsection{Preliminary results}
This section is devoted to the proof of the global well-posedness result stated in
theorem \ref{global} via the \textbf{I}-method.

Let $s\leq 0$ and $N>1$ be fixed. Let us define the Fourier multiplier operator 
\begin{equation}
	\widehat{I^{-s}_N u}(\xi)=\widehat{I u}(\xi)=m(\xi)\widehat{u}(\xi), \ m(\xi)=\left\lbrace \begin{array}{ll}
		1, & |\xi|<N, \\
		N^{-s}|\xi|^s, & |\xi|\geq 2N
	\end{array} \right.
\end{equation} 
where $m$  is a smooth non-negative function.   

\begin{lem}\label{aa}
	The operator $I$ applies $H^s(\mathbb{R})\longmapsto L^2 $. Moreover, the operator $I$ commute with differential operators and $\overline{Iu}=I\overline{u}$. That is,
	
	\begin{enumerate}
		\item $\left\| I(u)\right\|_{L^2}\leq c N^{-s}\left\| u\right\|_{H^s}$
		\item $P(D)I(u)= I\left( P(D)u\right)$,
	\end{enumerate}
	where $P$ is a polynomial and $D=\dfrac{d}{idx}$ is the differential operator.
	
\end{lem}

\begin{proof}
	It follows from the definition of $I$ and properties of the Fourier Transform.
\end{proof}
%
%
%
%
%
%

We will need the following

\begin{lem}[Lemma 12.1 of \cite{colliander-2004}]
	Let $\alpha_0>0$ and $n\geq 1$. Suppose $Z, \ X_1,\cdots,\ X_n$ are translation-invariant Banach spaces and $T$ is a translation invariant $n-$linear
	operator such that
	\begin{equation*}
		\left\|I^{\alpha}_1 T(u_1,\cdots, u_n) \right\|_Z \leq c \ \displaystyle\prod_{j=1}^n \left\| I_1^{\alpha}u_j\right\|_{X_j}, 
	\end{equation*} 
	for all $u_1, \cdots, \ u_n$, $0\leq \alpha\leq \alpha_0$. Then,
	\begin{equation*}
		\left\|I^{\alpha}_N T(u_1,\cdots, u_n) \right\|_Z \leq c \ \displaystyle\prod_{j=1}^n \left\| I_N^{\alpha}u_j\right\|_{X_j}, 
	\end{equation*} 
	for all $u_1, \cdots, \ u_n$, $0\leq \alpha\leq \alpha_0$ and $N\geq 1$. Here, the implied constant is independent of  $N$.
\end{lem}

Another essential result is

\begin{lem}[Lemma 5.1 of \cite{corcho-2009}]\label{chave}
	We have
	$$
	\left\|\left(D^{1/2}_xf\right)\cdot g \right\|_{L^2_{x,t}}\leq c \left\|f \right\|_{X^{0,1/2}} \left\|g \right\|_{X^{0,1/2}}, 
	$$	
	if $|\xi_2|\ll |\xi_1|$ for any $|\xi_1|\in\mbox{supp}\left(\widehat{f}\right)$ and $|\xi_2|\in\mbox{supp}\left(\widehat{g}\right)$. Moreover, this estimate is
	true if $f$ and/or $g$ is replaced by its complex conjugate in the left-hand side of the	inequality.
	
\end{lem}

\begin{remark}
	The lemma above is valid replacing $X^{0,1/2}$ by $X_a^{0,1/2}$.
\end{remark}

%

\subsection{Local well-posedness revisited}

Now, we take $N\gg 1 $  a sufficiently large integer and we denote by $I$ the operator $I := I_N^{-s}$ for a given $s \in \mathbb{R}$.

We have that the system (\ref{nls-g}) applied to the operator $I$ is given by 

\begin{equation}\label{sistema1}
	\left\lbrace \begin{array}{ll}
		i\partial_t Iu+p\partial^2_x Iu -\theta Iu+\ I\left(\overline{u} v\right) &=0  \\
		i\sigma\partial_t Iv+q\partial^2_x Iv -\alpha Iv+\frac{1}{2}I\left(u^2\right)&=0 
	\end{array} \right. .
\end{equation}

%
%

Let us state here a lemma that will be used to demonstrate the local well-posedness theorem, and then re-obtain the bilinear estimates.

\begin{lem}\label{chave1}
	Given $ -1/2< b'\leq b<1/2$, $s\in \mathbb{R}$, $a\geq 0$ and $0<T<1$, the estimate below 
	\begin{equation}
		\left\|\psi_T(t) u\right\|_{X_a^{s,b'}}\leq c T^{b-b'}\left\| u\right\|_{X_a^{s,b}}
	\end{equation} holds.
\end{lem}

\begin{proof}
See	\cite{ginibre-1997}.
\end{proof}

\begin{lem} If $1/4<d$, for $b_1,b_2\in \mathbb{R}$ such that $(b_1,b_2)=\left(0,\frac12+\right)$ or $(b_1,b_2)=\left(\frac12+,0\right)$, then
	\begin{equation}\label{b1}
		\| \overline{u}\cdot v\|_{X^{0,-d}}\leq c \|u\|_{X^{0,b_1}}\cdot \|v\|_{X_a^{0,b_2}}.
	\end{equation}
\end{lem}

\begin{proof}
	Without loss of generality, let us prove only the case $b_2=0$ and $b_1=\frac12+.$
	Following the ideas from Proposition \ref{p1}, it follows that
	
	\begin{eqnarray*}
		\| \overline{u}\cdot v\|_{X^{0,-d}}\leq \left\| u \right\|_{X^{0,b_1}}\left\|  v \right\|_{X^{0,b_1}} \left\| \frac{1}{\langle\tau_2+a\xi_2^2\rangle^{2b_2}}  \displaystyle\int_{\mathbb{R}^2}\frac{1}{\langle\tau_1-\xi_1^2\rangle^{2b_1}\langle\tau+\xi^2\rangle^{2d}}d\xi d\tau\right\|_{L^{\infty}_{\xi_2,\tau_2}}.
	\end{eqnarray*}
	
	On the right hand side of the inequality above, using the Lemma \ref{conta1} and the Lemma \ref{conta2}, we have that
	
	\begin{equation*}
		\displaystyle\int_{\mathbb{R}^2}\frac{1}{\langle\tau_1-\xi_1^2\rangle^{2b_1}\langle\tau+\xi^2\rangle^{2d}}d\xi d\tau\leq \displaystyle\int_{\mathbb{R}^2}\frac{1}{\langle\tau_2+2\xi^2+\xi_2^2-2\xi\xi_2\rangle^{2d}}d\xi d\tau\leq c.
	\end{equation*}
	
\end{proof}

Analogously, we prove the lemma below.
\begin{lem}\label{b2}
	Consider $1/4<d$. Given $b_1,b_2\in \mathbb{R}$ such that $(b_1,b_2)=\left(0,\frac12+\right)$ or $(b_1,b_2)=\left(\frac12+,0\right)$. Then
	\begin{equation}
		\| uw\|_{X^{0,-d}_a}\leq c \|u\|_{X^{0,b_1}}\cdot \|w\|_{X^{0,b_2}}.
	\end{equation}
\end{lem}	

\begin{remark}
	The above results are independent of the value of $a >0$.
\end{remark}

Now let us revisit the fixed-point theorem to find the best exponent for $\delta$.

\begin{prop}\label{local}
	For all $(u_0,v_0)\in H^s\times H^s$ and $s\geq -\frac{1}{4}$ and $0<a<\frac{1}{2}$ or $s\geq - \frac{1}{2}$ and $a>\frac{1}{2}$ the system  (\ref{sistema1}) has a unique local-in-time solution $(u(t),v(t))$ defined on the
	time interval $[0,\delta]$ for some $\delta\leq 1$ satisfying
	\begin{equation}
		\delta \sim \left(\left\|Iu_0 \right\|_{L^2_x} +\left\| Iv_0\right\|_{L^2_x} \right)^{-\frac{4}{3}+}.
	\end{equation}	
	
	Furthermore, $\left\|Iu_0 \right\|_{X^{0,1/2+}}+  \left\|Iv_0 \right\|_{X^{0,1/2+}_a} \leq c\left( \left\|Iu_0 \right\|_{L^2}+\left\|Iv_0 \right\|_{L^2}\right).$ 
\end{prop}

\begin{proof}
	Using the Lemmas\ref{aa}-\ref{b2} the proof follows in a similar to the Proposition 5.5 of \cite{corcho-2009}.
\end{proof}

\subsection{Almost conservation of the modified energy}

Let us consider the energy $E$ associated with the system(\ref{sistema1})
\begin{equation}\label{energia}
	E(Iu,Iv)=\left\|Iu \right\|_{L^2}^2+2\sigma\left\| Iv\right\|_{L^2}^2.
\end{equation}

\begin{teo} The functional energy (\ref{energia}) was derived with respect to the time given by:
	$$	\dfrac{d}{dt}E(Iu, Iv)= 2 \mbox{Im} \left\lbrace \displaystyle\int \left(I(\overline{u}v)-I\overline{u}Iv)\right) I\overline{u}dx\right\rbrace+2 \mbox{Im} \left\lbrace \displaystyle\int \left(I(u^2)-(Iu)^2\right) I\overline{v}dx\right\rbrace.$$
\end{teo}

\begin{proof}
	Also using the following fact $\displaystyle\int  \overline{f} \cdot\partial^2_x f= \displaystyle\int  \left|\partial_x f\right|^2$, we get:
	$$
	\begin{array}{ll}
	\dfrac{d}{dt}E(Iu, Iv)&= \displaystyle\int \partial_t Iu \cdot I\overline{u}+\displaystyle\int  Iu \cdot\partial_t I\overline{u}+2\sigma\displaystyle\int \partial_t Iv \cdot I\overline{v}+2\sigma\displaystyle\int  Iv \cdot\partial_t I\overline{v} \\
	& =-2\mbox{Im}\left\lbrace \displaystyle\int \left(I(\overline{u}v)-I\overline{u}Iv\right)\cdot I\overline{u}\right\rbrace+2\mbox{Im}\left\lbrace \displaystyle\int  \left( I(\overline{u}^2) -\left(I\overline{u}\right)^2\right)\cdot Iv  \right\rbrace.\\
	\end{array}
	$$
\end{proof}

From now on $\delta= \left( \left\|Iu\right\|_{L^2}+\left\|Iv\right\|_{L^2} \right)^{-4/3}$. Let us now estimate the modified energy. Using the fundamental theorem of calculus we have

\begin{eqnarray*}
	E(Iu,Iv)(\delta)-E(Iu,Iv)(0)&=&2 \mbox{Im} \displaystyle\int_0^{\delta}\left( \displaystyle\int \left(I(\overline{u}v)-I\overline{u}Iv)\right) I\overline{u}dx\right)dt\\
	&=&2 \mbox{Im} \displaystyle\int_0^{\delta} \bigg\langle \left(I(\overline{u}v)-I\overline{u}Iv\right)^{\wedge};\  \widehat{Iu}\bigg\rangle_{L^2}dt\\
	& &+2 \mbox{Im}\displaystyle\int_0^{\delta} \bigg\langle  \left(I(u^2)-(Iu)^2\right)^{\wedge}; \ \widehat{Iv}\bigg\rangle_{L^2}dt.
\end{eqnarray*}
Observe  that 
\begin{eqnarray*}
	\left(I(\overline{u}v)-I\overline{u}Iv\right)^{\wedge}&=& m(\xi)\widehat{\overline{u}\cdot v} -\widehat{I\overline{u}} \ast \widehat{Iv}\\
	&=&\displaystyle\int \widehat{I\overline{u}}(\xi_1)\ \widehat{Iv}(\xi_2)\left(\dfrac{m(\xi)-m(\xi_1)m(\xi_2) }{m(\xi_1)m(\xi_2)}\right)d\xi_1
\end{eqnarray*}
and
\begin{eqnarray*}
	\left(I(u^2)-(Iu)^2\right)^{\wedge}&=& m(\xi)\widehat{ u^2} -\widehat{Iu} \ast \widehat{Iu}\\
	&=&\displaystyle\int \widehat{Iu}(\xi_1)\ \widehat{Iu}(\xi_2)\left(\dfrac{m(\xi)-m(\xi_1)m(\xi_2) }{m(\xi_1)m(\xi_2)}\right)d\xi_1.
\end{eqnarray*}
Therefore,

\begin{eqnarray*}
	\displaystyle\int_0^{\delta} \bigg\langle \left(I(\overline{u}v)-I\overline{u}Iv\right)^{\wedge};\  \widehat{Iu}\bigg\rangle_{L^2}dt = 
	\displaystyle\int_0^{\delta}\displaystyle\int_{\mathbb{R}_{\xi}}\displaystyle\int_{\mathbb{R}_{\xi_1}}\widehat{I\overline{u}}(\xi_1)\ \widehat{Iv}(\xi_2)\ \widehat{I\overline{u}}(\xi)M(\xi,\xi_1)d\xi_1\ d\xi \ dt,
\end{eqnarray*}
analogously we have that

\begin{eqnarray*}
	\displaystyle\int_0^{\delta} \bigg\langle  \left(I(u^2)-(Iu)^2\right)^{\wedge}; \ \widehat{Iv}\bigg\rangle_{L^2}dt = 
	\displaystyle\int_0^{\delta}\displaystyle\int_{\mathbb{R}_{\xi}}\displaystyle\int_{\mathbb{R}_{\xi_1}}\widehat{Iu}(\xi_1)\ \widehat{Iu}(\xi_2)\ \widehat{I\overline{v}}(\xi)M(\xi,\xi_1)d\xi_1\ d\xi \ dt,
\end{eqnarray*}
where $M(\xi,\xi_1)=\left(\dfrac{m(\xi)-m(\xi_1)m(\xi_2) }{m(\xi_1)m(\xi_2)}\right)$.

\vspace{0.5cm}

We note that fixed $N>1$, $|\xi_1| \sim N_1$ and $|\xi_2| \sim N_2$, we have:
\begin{itemize}
	\item[(i)] If $2|\xi_1|\leq  |\xi_2|$ and $2|\xi_1|\leq  N$ then $|M(\xi,\xi_1)| \lesssim \frac{N_1}{N_2}$;
	\item[(ii)] If $2|\xi_2|\leq |\xi_1|$ and $2|\xi_2|\leq N$ then $|M(\xi,\xi_1)| \lesssim \frac{N_2}{N_1}$;
	\item[(iii)] If $2|\xi_1|\leq |\xi_2|$ and $|\xi_1|\geq 2N$ then $|M(\xi,\xi_1)| \lesssim \frac{N_1}{N}$;
	\item[(iv)] If $2|\xi_2|\leq |\xi_1|$ and $|\xi_2|\geq  2N$ then $|M(\xi,\xi_1)| \lesssim \frac{N_2}{N}$ and 
	\item[(v)] If $|\xi_1|\sim |\xi_2|\gtrsim N$ then $|M(\xi,\xi_1)| \lesssim \left(\frac{N_1}{N}\right)^{2}$.
\end{itemize}

By the symmetry of the variables it is sufficient to verify only the statements (i), (iii) and (v).

We will use the fact that $m'(\xi)=-N|\xi|^{-2}$.

In the first case, as  $|\xi_1|\ll N$, we get $m(\xi_1)=1$, hence
\begin{equation*}
	\left|M(\xi,\xi_1)\right|=\left|\dfrac{m(\xi_1+\xi_2)-m(\xi_2)}{m(\xi_2)}\right|\sim \left|\dfrac{m'(\xi_2)|\xi_1|}{m(\xi_2)}\right|\lesssim \dfrac{N_1}{N_2}.
\end{equation*}

Still, to verify the item (iii) we observe that $\frac12 |\xi_2|\leq |\xi_1+\xi_2|\leq 2|\xi_2|$ and thereby,

\begin{eqnarray*}
	\dfrac{m(\xi_1+\xi_2)-m(\xi_1)m(\xi_2)}{m(\xi_2)}&&= \dfrac{N|\xi_1+\xi_2|^{-1}-N|\xi_2|^{-1}N|\xi_1|^{-1}}{N|\xi_2|^{-1}}\\
	&& =\dfrac{|\xi_2|}{|\xi_1+\xi_2|}-\dfrac{N}{|\xi_1|}\\
	&& \leq 2 - \dfrac{N}{|\xi_1|}\sim 1.
\end{eqnarray*}
Then, (iii) follows easily from observation that $M(\xi,\xi_1)\sim \dfrac{1}{m(\xi_1)}=\dfrac{N_1}{N}$.

The last case, follows from the fact that 
\begin{eqnarray*}
	m(\xi_1+\xi_2)-m(\xi_1)m(\xi_2)&& = N|\xi_1+\xi_2|^{-1}-N^2|\xi_1|^{-1}|\xi_2|^{-1}\\
	&& \sim N\left(\dfrac{1}{2|\xi_1|}-\dfrac{N}{|\xi_1|^2}\right)\\
	&& = \dfrac{N}{2|\xi_1|}\dfrac{|\xi_1|-2N}{|\xi_1|}\sim 1.
\end{eqnarray*}

Therefore, $M(\xi,\xi_1)\sim \frac{1}{m(\xi_1)m(\xi_2)}\sim \left(\dfrac{N_1}{N}\right)^2$.

Considering
\begin{equation}
	L_1= 2\mbox{Im}\displaystyle\int_0^{\delta}\displaystyle\int_{\mathbb{R}_{\xi}}\displaystyle\int_{\mathbb{R}_{\xi_1}}\widehat{I\overline{u}} (\xi_1)\  \widehat{Iv}(\xi_2)\ \widehat{I\overline{u}}(\xi)M(\xi,\xi_1)d\xi_1\ d\xi \ dt
\end{equation}
and 

\begin{equation}
	L_2= 2 \mbox{Im}\displaystyle\int_0^{\delta}\displaystyle\int_{\mathbb{R}_{\xi}}\displaystyle\int_{\mathbb{R}_{\xi_1}}\widehat{Iu}(\xi_1)\ \widehat{Iu}(\xi_2)\ \widehat{I\overline{v}}(\xi)M(\xi,\xi_1)d\xi_1\ d\xi \ dt,
\end{equation}
we get,

$$
\left|E(Iu,Iv)(\delta)-E(Iu,Iv)(0)\right|=|L_1+L_2|.
$$

\begin{prop}\label{pp1}
	For $\sigma>2$ and $s\geq -1/2$ we have 
	\begin{equation}
		\left|E(Iu,Iv)(\delta)-E(Iu,Iv)(0)\right| \leq N^{-\frac12}\delta^{\frac12}\left\|I(u) \right\|_{X^{0,\frac{1}{2}+}}^2\left\|I(v) \right\|_{X^{0,\frac{1}{2}+}}. 
	\end{equation}
\end{prop}

\begin{proof} 
	It is enough to estimate $L_1$ and $L_2$. We still note  that $L_1$ and $L_2$ are equivalent. In this case, let us restrict ourselves to estimating $L_1$. Let us use the notation $|\xi|=|\xi_1+\xi_2|\sim N_3$
	
	For $2|\xi_1|\leq  |\xi_2|$ and $2|\xi_1|\leq  N$ such that $|M(\xi,\xi_1)| \lesssim \frac{N_1}{N_2}$. Then, from Lemmas \ref{chave} and \ref{chave1}, we see that
	
	\begin{eqnarray*}
		|L_1| &&\leq \left(\frac{N_1}{N_2}\right)^{1/2} \left\| D_x^{1/2}\widehat{I\overline{u}} (\xi_1)\cdot  \widehat{Iv}(\xi_2)\right\|_{L^2}\left\| \widehat{I\overline{u}}\right\|_{L^2}\\
		&& \leq \left(\frac{N_1}{N_2}\right)^{1/2} N_3^{-1/2} \left\|\widehat{Iu}\right\|_{X^{0,1/2+}} \left\| \widehat{Iv}\right\|_{X^{0,1/2+}}\delta^{1/2} \left\| \widehat{Iu}\right\|_{X^{0,1/2+}}\\
		&& \leq N^{-1/2} \delta^{1/2}\left\|I(u) \right\|_{X^{0,\frac{1}{2}+}}^2\left\|I(v) \right\|_{X^{0,\frac{1}{2}+}}.
	\end{eqnarray*}
	
	The case (ii), that is, $2|\xi_2|\leq  |\xi_1|$ and $2|\xi_2|\leq  N$ follows by the symmetry of the variables.
	\vspace{0.5cm}

	In the prove of cases (iii) and (iv)  
	when $s=-1/2$, such that $|M(\xi,\xi_1)| \lesssim \left(\frac{N_1}{N}\right)^{1/2}$
	
	\begin{eqnarray*}
		|L_1|&& \leq \left(\frac{N_1}{N}\right)^{1/2}\left\| D_x^{1/2}\widehat{I\overline{u}} (\xi_1)\cdot  \widehat{Iv}(\xi_2)\right\|_{L^2}\left\| \widehat{I\overline{u}}\right\|_{L^2}\\
		&& \leq \left(\frac{N_1}{N}\right)^{1/2} N_3^{-1/2} \left\|\widehat{Iu}\right\|_{X^{0,1/2+}} \left\| \widehat{Iv}\right\|_{X^{0,1/2+}}\delta^{1/2} \left\| \widehat{Iu}\right\|_{X^{0,1/2+}}\\
		&& \leq N^{-1/2} \delta^{1/2}\left\|I(u) \right\|_{X^{0,\frac{1}{2}+}}^2\left\|I(v) \right\|_{X^{0,\frac{1}{2}+}}.
	\end{eqnarray*}
	
	For the last case, we have $|M(\xi,\xi_1)| \lesssim \frac{N_1}{N}$ when $|\xi_1|\sim |\xi_2|\gtrsim N$ thereby,  $|\xi_1|\leq 2 |\xi|$ and it implies
	
	\begin{eqnarray*}
		|L_1|&& \leq \frac{N_1}{N}\left\| D_x^{1/2}\widehat{I\overline{u}} (\xi_1)\cdot  \widehat{Iv}(\xi_2)\right\|_{L^2}\left\| \widehat{I\overline{u}}\right\|_{L^2}\\
		&& \leq \frac{N_1}{N} N_1^{-1/2} \left\|\widehat{Iu}\right\|_{X^{0,1/2+}} \left\| \widehat{Iv}\right\|_{X^{0,1/2+}}\delta^{1/2} \left\| \widehat{Iu}\right\|_{X^{0,1/2+}}\\
		&& \leq N^{-1} \delta^{1/2}\left\|I(u) \right\|_{X^{0,\frac{1}{2}+}}^2\left\|I(v) \right\|_{X^{0,\frac{1}{2}+}}.
	\end{eqnarray*}

	Since $ \left|E(Iu,Iv)(\delta)-E(Iu,Iv)(0)\right|=|L_1+L_2|\leq |L_1|+|L_2|\leq c |L_1|$ we obtain  $$\left|E(Iu,Iv)(\delta)-E(Iu,Iv)(0)\right|\leq cN^{-1} \delta^{1/2}\left\|I(u) \right\|_{X^{0,\frac{1}{2}+}}^2\left\|I(v) \right\|_{X^{0,\frac{1}{2}+}}.$$
	
\end{proof}

Following the same arguments presented above, we prove the following
\begin{prop}
	For $0<\sigma<2$ and $s\geq -1/4$  have 
	\begin{equation}
		\left|E(Iu,Iv)(\delta)-E(Iu,Iv)(0)\right| \leq N^{-\frac14}\delta^{\frac12}\left\|I(u) \right\|_{X^{0,\frac{1}{2}+}}^2\left\|I(v) \right\|_{X^{0,\frac{1}{2}+}}. 
	\end{equation}
\end{prop}

\begin{proof}
	Analogous to the previous case.
\end{proof}

\subsection{ Global existence}
In this subsection we will demonstrate Theorem \ref{global}.

\begin{proof}
	Given the initial conditions of the Cauchy Problem (\ref{nls-g}) $(u_0, v_0)\in H^s\times H^s$ such that 
	$$
	\left\|I(u_0) \right\|_{L^2}\leq c N^{-s}\left\|\right\|_{H^s} \ \mbox{ and } \ \left\|I(v_0) \right\|_{L^2}\leq c N^{-s}\left\|v_0\right\|_{H^s}.
	$$
	
	Applying the local well-posedness result of the Proposition \ref{local}, we see that there exists a unique solution in the time interval $[0, \delta]$, where $\delta \sim N^{-4s/3}$ and such that 
	$$
	\left\|I(u) \right\|_{X^{0,\frac12+}}+\left\|I(v) \right\|_{X^{0,\frac12+}}\leq c N^{-s}.
	$$ 
	
	For $\sigma>2$ and $s\geq -\frac12$ and using the Proposition \ref{pp1}, we have  
	
	$$
	\left|E(Iu,Iv)(\delta)-E(Iu,Iv)(0)\right| \leq N^{-\frac12}\delta^{\frac12}N^{-3s}.
	$$
	
	We should now prove that for every $T>0$ we can extend our solution to the range $[0,T]$. In order to do it, it is enough to apply the local well-posedness Theorem\ref{local} until we reach this interval, that is, $T/\delta$ times. If the modified energy does not grow more than the initial one for this number of interactions we can conclude that the result is extended up to the interval $[0,T]$, that is, we should have
	
	\begin{equation}
		\left|E(Iu,Iv)(\delta)-E(Iu,Iv)(0)\right|\frac{T}{\delta}\ll E(Iu_0,Iv_0).
	\end{equation} 
	Therefore, it is sufficient that 
	\begin{equation}
		N^{-\frac12}\delta^{\frac12}N^{-3s}\frac{T}{\delta}\ll N^{-2s} \ \mbox{ or } \ N^{-\frac12}\delta^{-\frac12}N^{-3s}T\ll N^{-2s}.
	\end{equation}
	
	Hence we conclude that $-\frac12 -3s+\frac{2s}{3}\leq -2s$ because $\delta^{-1/2}\sim N^{2s/3}$.
	It turns out that for any $s\geq-1/2$ 
	we can extend the solution at any time interval by taking $1\ll N$.
	
	\vspace{0.5cm}
	
	The prove of the other case ($0<\sigma<2$) follows similarly.
	
\end{proof}
The Theorem {showed} in this section tells us that the solution of the Cauchy Problem extends globally, in time, {in the sense that it connects} the points $(0,0)$ and $(-1/2,-1/2)$ {when} $\sigma>2$ and the points $(0,0)$ and $(-1/4,-1/4)$ in the case $0<\sigma<2$.

\section*{Acknowledgments}

{This paper is part of my Ph.D. thesis at the Federal University of Alagoas under the guidance of my advisor Adán J. Corcho. I want to take the opportunity	to express my sincere gratitude to him.


	
	\bibliographystyle{abbrv}
	\addcontentsline{toc}{section}{References}
	\bibliography{referencias}
\end{document}